\documentclass[a4paper,12pt,final]{amsart}

\usepackage[pdfauthor={Eugen Stumpf},%
pdftitle={On a delay differential equation arising from a car-following model: wavefront solutions with constant-speed and their stability},%
colorlinks=true]{hyperref}

\usepackage{tikz}

\newtheorem{theorem}{Theorem}[section]
\newtheorem{proposition}[theorem]{Proposition}

\newtheorem{corollary}[theorem]{Corollary}
\newtheorem{conjecture}[theorem]{Conjecture}
\theoremstyle{definition}

\theoremstyle{remark}

\newtheorem{remark}[theorem]{Remark}

\numberwithin{equation}{section}
\newcommand{\R}{\mathbb{R}}  

{}
\copyrightinfo{2016}{Eugen Stumpf}

\PII{}

\title[Wavefront solutions with constant-speed and their stability] 
{On a delay differential equation arising from a car-following model: Wavefront solutions with constant-speed and their stability}
\author{Eugen Stumpf}
\address{Department of Mathematics, University of Hamburg, Bundesstrasse 55, 20146 Hamburg, Germany}
\email{eugen.stumpf@math.uni-hamburg.de}
\urladdr{www.math.uni-hamburg.de/home/stumpf/index\_en.html}
\date{September 2016}

\begin{document}

\begin{abstract}
	This work is concerned with the study of the scalar delay differential equation
	\begin{equation*}
	z^{\prime\prime}(t)=h^2\,V(z(t-1)-z(t))+h\,z^\prime(t)
	\end{equation*}
	motivated by a simple car-following model on an unbounded straight line. Here, the positive real $h$ denotes some parameter, and $V$ is a so-called \textit{optimal velocity function} of the traffic model involved. We analyze the existence and local stability properties of solutions $z(t)=c\,t+d$, $t\in\R$, with $c,d\in\R$. In the case $c\not=0$, such a solution of the differential equation forms a wavefront solution of the car-following model where all cars are uniformly spaced on the line and move with the same constant velocity. In particular, it is shown that all but one of these wavefront solutions are located on two branches parametrized by $h$. Furthermore, we prove that along the one branch all solutions are unstable due to the principle of linearized instability, whereas along the other branch some of the solutions may be stable. The last point is done by carrying out a center manifold reduction as the linearization does always have a zero eigenvalue. Finally, we provide some numerical examples demonstrating the obtained analytical results.
\end{abstract}
\maketitle
\section{Introduction}
Consider a system of countably infinite many cars moving one after another from the left to the right-hand side along a single lane road which we shall identify with the real line in the following. After fixing some origin on the road, and thus specifying the origin on the real line, we label the cars by the integers and denote the position of each car $j\in\mathbb{Z}$ at time $t\in\R$ relative to the origin by the coordinate $x_j(t)\in\R$ as indicated in Figure \ref{fig: setting}. 

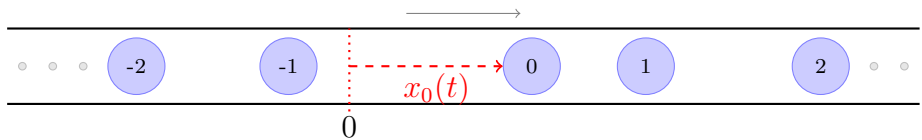
\begin{figure}[h]
	\vspace*{0.5cm}
	\begin{tikzpicture}
	\draw[thick] (0,0) -- (12.,0);
	\draw[thick] (0,1) -- (12,1);
	\draw[-,red,thick,dotted] (4.5,-0.1) -- (4.5,1);
	\node (auto-2) at (1.7,0.5) [circle,minimum size=.75cm,draw=blue!50,fill=blue!20] {\tiny -2};
	\node at (3.7,0.5) [circle,minimum size=0.75cm,draw=blue!50, fill=blue!20]{\tiny -1};
	\node (auto 0) at (6.9,0.5) [circle,minimum size=0.75cm, draw=blue!50, fill=blue!20]{\tiny 0};
	\node at (8.4,0.5) [circle,minimum size=0.75cm,draw=blue!50, fill=blue!20]{\tiny 1};
	\node at (10.7,0.5) [circle,minimum size=0.75cm, draw=blue!50, fill=blue!20]{\tiny 2};
	\node at (0.2,0.5) [circle, inner sep=1pt,draw=gray!50, fill=gray!20]{};
	\node at (0.6,0.5) [circle, inner sep=1pt,draw=gray!50, fill=gray!20]{};
	\node at (1,0.5) [circle, inner sep=1pt, draw=gray!50, fill=gray!20]{};
	\node at (11.8,0.5) [circle, inner sep=1pt,draw=gray!50, fill=gray!20]{};
	\node at (11.4,0.5) [circle, inner sep=1pt,draw=gray!50, fill=gray!20]{};
	\draw [->,gray] (5.25,1.2) -- (6.75,1.2);
	\draw[->,dashed,thick,red] (4.5,0.5) -- (auto 0.west);
	\node at (5.65,0.2){\textcolor{red}{$x_0(t)$}};
	\node at (4.5,-0.3){\textcolor{black}{$0$}};
	\end{tikzpicture}
	\caption{The schematic setting of the car-following model}\label{fig: setting}
\end{figure}
Let us assume that each driver attempts to drive with a velocity according to some optimal velocity which depends only on the headway, that is, on the distance to the car in front. Then, considering the most simple case where the \textit{optimal velocity function} $V:\R\to[0,\infty)$ is the same for each driver, the motion of the cars along the line is given by the system
\begin{equation}\label{eq: traffic_model}
x^{\prime\prime}_{j}(t)=V(x_{j+1}(t)-x_j(t))-x^{\prime}_j(t),\qquad j\in\mathbb{Z}.
\end{equation}
of coupled ordinary differential equations.
For the optimal velocity function $V$ we make the following standing assumptions:
\begin{itemize}
	\item[\textbf{(OVF 1)}] $V$ is non-negative and monotonically increasing.
	\item[\textbf{(OVF 2)}] $V$ is bounded from above by some maximum velocity $V^{\max}>0$ and
	\begin{equation*}
	\lim_{s\to\infty}V(s)=V^{\max}.
	\end{equation*}
	\item[\textbf{(OVF 3)}] There is a \emph{safety distance} $d_S\geq 0$ such that $V(s)=0$ for all $ s\leq d_S$ and $V(s)>0$ as $s>d_S$.
	\item[\textbf{(OVF 4)}] $V$ is $C^1$-smooth, twice continuously differentiable in $(d_S,\infty)$, and there is some constant $b>0$ such that $V^{\prime}$ is strictly increasing in $(d_S,b)$ and, on the other hand, strictly decreasing in $(b,\infty)$. 
\end{itemize}
An example of such a function is given by
\begin{equation}\label{eq: OVF}
V_q(s):=\begin{cases}
V^{\max}\,\displaystyle{\frac{(s-d_S)^2}{1+(s-d_S)^2}},& \text{for }s\geq d_S,\\
0,&\text{for }s<d_S,
\end{cases}
\end{equation}
with some fixed maximum velocity $V^{\max}>0$ and safety distance $d_S\geq 0$, and the typical shape of $V$ and $V^\prime$ is indicated in Figure \ref{fig: optVel}.
\begin{figure}[ht]
	\includegraphics[width=6cm]{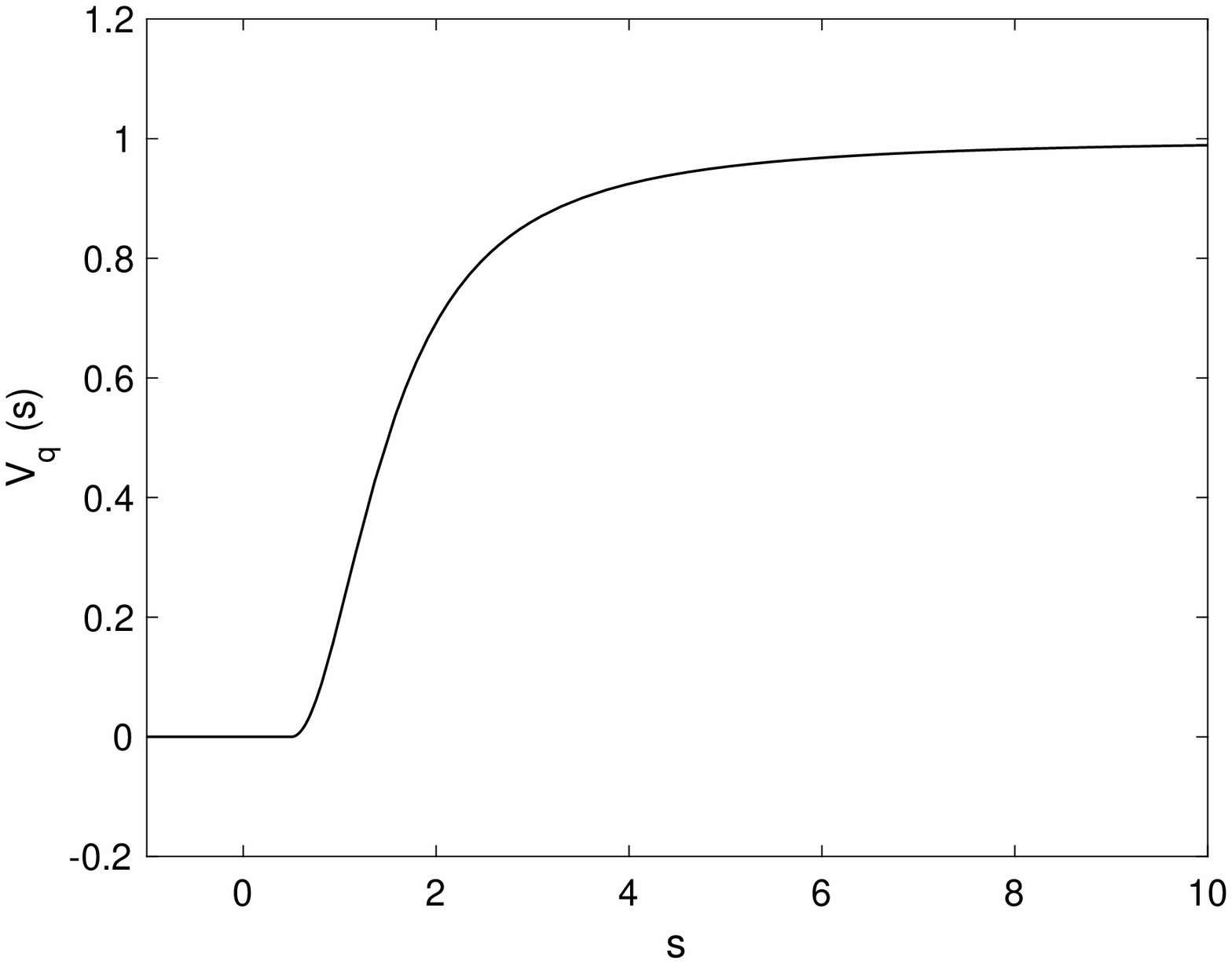}
	\includegraphics[width=6cm]{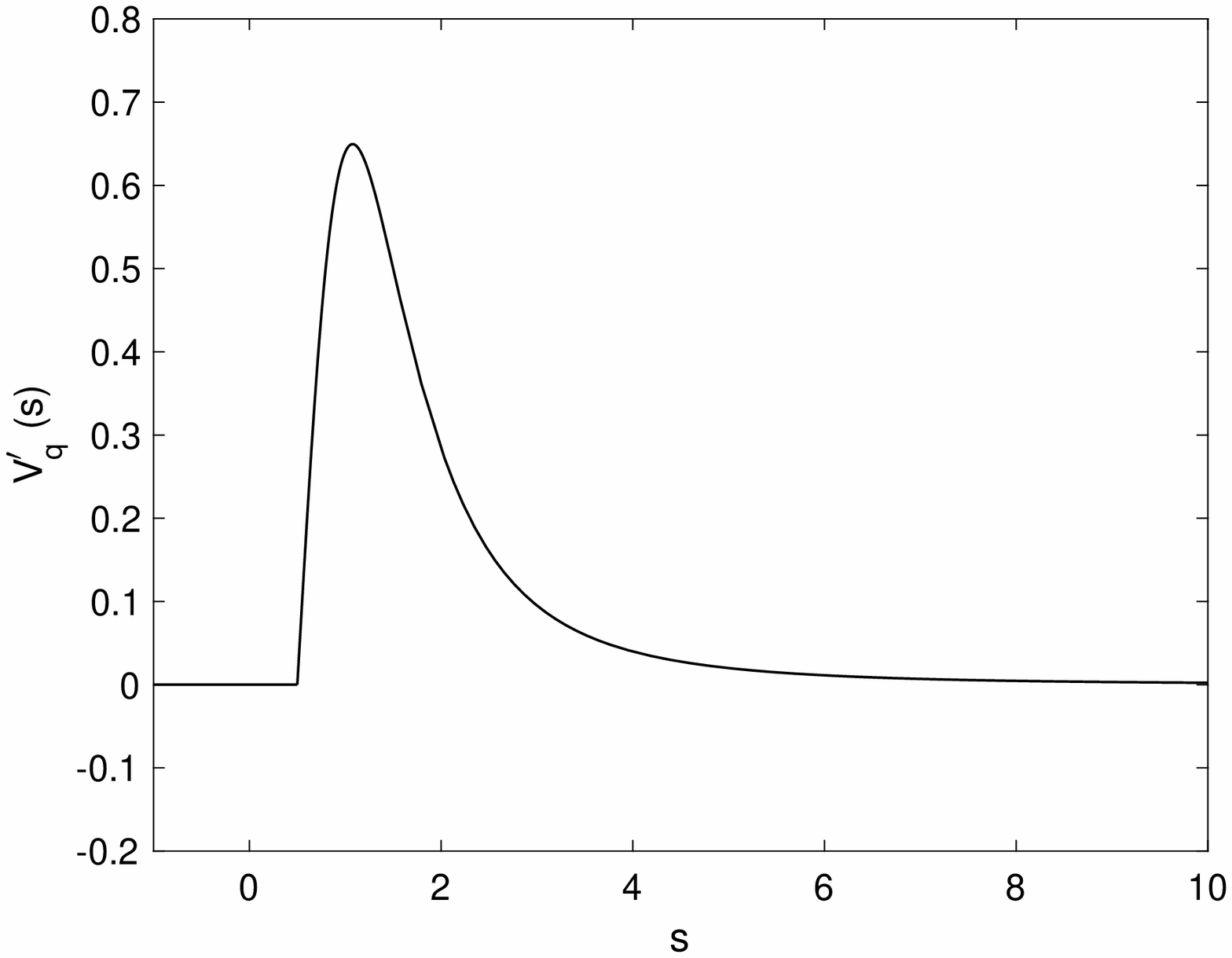}
	\caption{Function $V_q$ and its derivative for $V^{\max}=1$ and $d_S=0.5$}\label{fig: optVel}
\end{figure}

Now, suppose that, given some fixed parameter $h>0$, there exists a globally defined solution $z:\R\to\R$ of the scalar differential equation
\begin{equation}\label{eq: DDE}
z^{\prime\prime}(t)=h^2\,V(z(t-1)-z(t))+h\,z^{\prime}(t)
\end{equation}
with constant delay. Then a straightforward calculation yields that the family $\lbrace x_j\rbrace_{j\in\mathbb{Z}}$ of real-valued functions $x_j:\R\to\R$ defined by
\begin{equation}\label{eq: definition_of_wavefront_solutions}
x_j(t):=z\left(-\tfrac{1}{h}\,t-j\right),\qquad t\in\R,
\end{equation}
satisfies Eq. \eqref{eq: traffic_model}. Furthermore, by claiming $z^\prime(t)<0$ for all $t\in\R$, we obtain a solution of the traffic model \eqref{eq: traffic_model} which is characterized by the property that each driver acts in the same way as the driver of the car in front, but probably some time units later. Thus, for each parameter $h>0$ a strictly decreasing global solution of the delay differential equation \eqref{eq: DDE}  forms a particular wavefront solution of the proposed traffic model \eqref{eq: traffic_model}.

\begin{remark}
	1. Observe that a globally defined but not strictly monotonically decreasing solution $z$ of Eq. \eqref{eq: DDE} may generally result in a physically non-reasonable solution of the traffic model \eqref{eq: traffic_model}. For instance, assuming $z^\prime(t)>0$ for all $t\in\R$, the wavefront ansatz \eqref{eq: definition_of_wavefront_solutions} leads to
	\begin{equation*}
	x_{j+1}(t)=z\left(-\tfrac{1}{h}t-(j+1)\right)<z(-\tfrac{1}{h}t-j)=x_j(t)
	\end{equation*}
	as $t\in\R$. But in our setting the car with the number $j+1$ moves in front of the car with the number $j$ along the real line.
	
	2.  The wavefront ansatz \eqref{eq: definition_of_wavefront_solutions} would also work in the situation of a negative parameter $h<0$ involved in the delay differential equation \eqref{eq: DDE}. But again we would obtain a non-reasonable solution as it would result in
	\begin{equation*}
	x^\prime_j(t)=\frac{d}{dt}z\left(-\tfrac{1}{h}t-j\right)=\underbrace{\left(-\tfrac{1}{h}\right)}_{>0}\,\underbrace{z^\prime\left(-\tfrac{1}{h}t-j\right)}_{<0}<0.
	\end{equation*}
	Therefore, the cars would move from the right to left hand-side, in contrast to our assumption that the cars move from the left to the right-hand side.
\end{remark}

The main purpose of this paper is an analytical study of the delay differential equation \eqref{eq: DDE} for parameter $h>0$ with respect to the existence and local stability of so-called \textit{quasi-stationary solutions}, that is, solutions whose first derivative is constant.
In the case of a negative derivative, such a solution leads to a \emph{wavefront solution with constant-speed} of the traffic model \eqref{eq: traffic_model} where all cars are uniformly spaced on the line and move with the same velocity for all time. In detail, using elementary arguments, we will show that all but one of these solutions are located on two branches parametrized by the involved parameter $h>0$. Furthermore, we will prove that along the one branch all the wavefront solutions with constant-speed are unstable, whereas along the other branch some of those may be stable but not asymptotically stable. This will be done by applying, on the one hand, the principle of linearized instability, and on the other hand by employing a center manifold reduction as the linearization along a wavefront solution with constant-speed has always a zero eigenvalue.

The traffic model introduced above is a modification of a well-known car-following model describing the dynamics of $N\in\mathbb{N}$ cars moving on a circular single lane road with some fixed circumference of length $L>0$. The last mentioned model was introduced by Bando et al. in \cite{Bando1994, Bando1995b}, and then the original model as well as different modifications were extensively studied during the last twenty years. However, as in this work we will neither discuss the dynamical behavior of the model \eqref{eq: traffic_model} in the main, nor the obtained results from the traffic flow point of view, we refrain from a deeper discussion on related car-following models and results from the traffic flow theory. In particular, such a discussion would exceed the scope of this paper. But all these issues will be addressed in a later work \cite{Gasser2016} which is in progress.

The rest of this paper is organized as follows. The next section contains some preliminaries. Here, we rewrite Eq. \eqref{eq: DDE} in the more abstract form of a so-called \emph{retarded functional differential equation} and discuss some basic facts. Section \ref{sec: existence} deals with the existence of wavefront solutions with constant-speed, whereas in Section 
\ref{sec: linearization} we examine the linearization and its spectral properties along some fixed wavefront solution with constant-speed. Section \ref{sec: stability} is devoted to the study of the local stability properties of wavefront solutions with constant-speed, and in the final section we close this work with some numerical examples.
\section{Preliminaries}
From now on, let $\|\cdot\|_{\R^2}$ denote the Euclidean norm on $\R^2$ and $C$ the Banach space of all continuous functions $\varphi:[-1,0]\to\R^2$ equipped with the usual norm $\|\varphi\|_C=\sup_{-1\leq s\leq 0}\|\varphi(s)\|_{\R^2}$ of uniform convergence. Given $t\in\R$, an interval $I\subset \R$ with $[t-1,t]\subset I$, and some continuous function $w:I\to\R^2$, let the \emph{segment} $w_t\in C$ of $w$ at $t$ be defined by $w_t(s)=w(t+s)$ as $-1\leq s\leq 0$. 

Using the above notation of segments, and 
\begin{equation}\label{eq: definition reduction}
w(t):=\begin{pmatrix}
z(t)\\z^\prime(t)
\end{pmatrix}
\end{equation}
as the new state variable, Eq. \eqref{eq: DDE} for the special wavefront solutions of the traffic model takes the more convenient equivalent form
\begin{equation}\label{eq: DDE-canonical}
w^\prime(t)=f(w_t)
\end{equation}
where the right-hand side is defined by
\begin{equation}\label{eq: definition-F}
f:C\ni\varphi \mapsto \begin{pmatrix}
\varphi_2(0)\\ 
h^2\, V\left(\varphi_1(-1)-\varphi_1(0)\right)+h\, \varphi_2(0)
\end{pmatrix}\in\R^2.
\end{equation}
We have  $f(0)=0\in\R^2$ due to assumption (OVF 3), and the map $f$ is invariant with respect to translations into the direction of the constant function
\begin{equation}\label{eq: center-direction}
\hat{e}_1:[-1,0]\ni t\mapsto \begin{pmatrix}
1\\0
\end{pmatrix}\in\R^2
\end{equation} 
since apparently
\begin{equation}\label{eq: translation_invariance}
f(\varphi+k\,\hat{e}_1)=f(\varphi)
\end{equation}
for all $\varphi\in C$ and all $k\in\R$. Further, observe that $f$ is at least $C^1$-smooth. Indeed, introducing the two evaluation operators
\begin{equation*}
\mathrm{ev}_{0}:C\ni \varphi\mapsto \varphi(0)\in\R^2\qquad\text{and}\qquad \mathrm{ev}_{-1}:C\ni\varphi\mapsto\varphi(-1)\in\R^2,
\end{equation*} 
which both are continuous and linear,
and the map
\begin{equation*}
G:\R\ni\begin{pmatrix}
v_1\\v_2\\v_3\\v_4
\end{pmatrix}\mapsto \begin{pmatrix}
v_2\\
h^2\, V(v_3-v_1)+h\,v_2
\end{pmatrix}\in\R^2
\end{equation*}
which, in view of assumption (OVF 4), is continuously differentiable, the map $F$ defined by \eqref{eq: definition-F} may be written as the composition
\begin{equation*}
f=G\circ(\mathrm{ev}_0\times \mathrm{ev}_{-1})
\end{equation*}
of $C^1$-smooth maps. Hence, $f$ is continuously differentiable, and thus particularly satisfies a local Lipschitz-condition at each $\varphi\in C$. 

Under a solution of Eq. \eqref{eq: DDE-canonical} we understand either a continuously differentiable function $w:\R\to\R^2$ satisfying Eq. \eqref{eq: DDE-canonical} for all $t\in\R$, or a continuous function $w:[t_0-1,t_+)\to\R^2$, $t_0<t_+$, such that $w$ is continuously differentiable for all $t_0<t<t_+$ and satisfies Eq. \eqref{eq: DDE-canonical} as $t_0<t<t_+$.  For instance, combining $f(0)=0$ and the translation invariance \eqref{eq: translation_invariance}, we immediately see that for each real $d$ the constant function
\begin{equation*}
w^{0,d}:\R\ni t\mapsto \begin{pmatrix}
d\\0
\end{pmatrix}\in\R^2
\end{equation*}
forms a globally defined solution of Eq. \eqref{eq: DDE-canonical}. Moreover, if $w$ is another solution of Eq. \eqref{eq: DDE-canonical} then the restriction of the sum $w+w^{0, d}$ to the domain of $w$ is a solution of Eq. \eqref{eq: DDE-canonical} as well. In fact, using Eq. \eqref{eq: translation_invariance} reflecting the translation invariance we get
\begin{equation*}
\left(w+w^{0,d}\right)^\prime(t)
=w^{\prime}(t)
=f(w_t)=f\left(w_t+d\,\hat{e}_1\right)=f\left(w_t+w^{0,d}_t\right)
\end{equation*}
for all relevant $t$.
\begin{remark}\label{rem: zero speed}
	Note that for each $d\in\R$ the solution $w^{0, d}$ of Eq. \eqref{eq: DDE-canonical} leads to a physically non-reasonable solution of the traffic model given by Eq. \eqref{eq: traffic_model}. Indeed, the positions of all cars collapse to a single point $d$ on the road.
\end{remark}

But Eq. \eqref{eq: DDE-canonical} has much more solutions than those of type $w^{0, d}$ as we shall show next.
\begin{proposition}
	Each $\varphi\in C$ uniquely defines a solution $w^{\varphi}:[-1,\infty)\to\R^2$ of Eq. \eqref{eq: DDE-canonical} with $w_0^{\varphi}=\varphi$.	
\end{proposition}
\begin{proof}
	1. Recall that  the map $f$ is locally Lipschitz-continuous. Hence, following the basic existence theory for delay differential equations as, for instance, contained in Hale and Verduyn Lunel \cite{Hale1993} or in  Diekmann et al. \cite{Diekmann1995}, we see that for each initial function $\varphi\in C$ there exist a uniquely determined constant $t_+(\varphi)>0$ and an in the forward time-direction non-continuable solution $w^{\varphi}:[-1,t_+(\varphi))\to\R^2$ of Eq. \eqref{eq: DDE-canonical} with $w_0^{\varphi}=\varphi$. So, the only point remaining to prove is that $t_+(\varphi)=\infty$ for all $\varphi\in C$.
	
	2. Observe that for all $\psi\in C$ we have
	\begin{equation*}
	\begin{aligned}
	\|f(\psi)\|_{\R^2}&=\left\|\begin{pmatrix}
	\psi_2(0)\\h^2\,V(\psi_1(-1)-\psi_1(0))+h\,\psi_2(0)
	\end{pmatrix}\right\|_{\R^2}\\
	&\leq \left\|\begin{pmatrix}
	\psi_2(0)\\h\,\psi_2(0)
	\end{pmatrix}\right\|_{\R^2}+\left\|\begin{pmatrix}
	0\\h^2\,V(\psi_1(-1)-\psi_1(0))
	\end{pmatrix}\right\|_{\R^2}\\
	&=\sqrt{1+h^2}\,|\psi_2(0)|+h^2\,V(\psi_1(-1)-\psi_1(0))\\
	&\leq \sqrt{1+h^2}\|\psi\|_{C}+h^2\,\max_{\xi\in\R}V^{\prime}(\xi)\,(\psi_1(-1)-\psi_1(0))\\
	&\leq \sqrt{1+h^2}\|\psi\|_{C}+h^2\,V^{\prime}(b)\,(\|\psi\|_C+\|\psi\|_C)\\
	&=\left(\sqrt{1+h^2}+2h^2\,V^\prime(b)\right)\|\psi\|_C
	\end{aligned}
	\end{equation*}
	with constant $b$ from assumption (OVF 4). 
	
	3. Now, let $\varphi\in C$ be given. Set $w=w^{\varphi}:[-1,t_+(\varphi))\to\R^2$ for the associated maximal solution of Eq. \eqref{eq: DDE-canonical} due to the first part. Then, integrating Eq. \eqref{eq: DDE-canonical} and using the triangle inequality along with the part above, we obtain
	\begin{equation*}
	\begin{aligned}
	\|w(t)\|_{\R^2}&=\left\|w(0)+\int_{0}^{t}f(w_s)\,ds\right\|_{\R^2}\\
	&\leq \|w(0)\|_{\R^2}+\int_0^{t}\|f(w_s)\|_{\R^2}ds\\
	&\leq \|\varphi\|_C+\int_0^tK\|w_s\|_{C}ds\\
	&=\|\varphi\|_C+K\int_0^t\|w_s\|_{C}ds
	\end{aligned}
	\end{equation*}
	for all $0\leq t<t_+(\varphi)$ where the constant $K>0$ is given by
	\begin{equation*}
	K:=\sqrt{1+h^2}+2h^2\,V^\prime(b).
	\end{equation*}
	Furthermore, a straightforward argument now shows that
	\begin{equation*}
	\|w_t\|_{C}\leq \|\varphi\|_C+K\int_{0}^{t}\|w_s\|_C ds
	\end{equation*}
	as long as $0\leq t<t_+(\varphi)$. Consequently, applying the Gronwall's inequality, we see that
	\begin{equation*}
	\|w_t\|_C\leq \|\varphi\|_C e^{Kt}
	\end{equation*}
	for all $0\leq t<t_+(\varphi)$. 
	
	4.  Observe that, in view of the second part of the proof, $f$ maps bounded sets of $C$ into relative compact sets of $\R^2$. Therefore, basic continuation results -- see, for instance Hale and Verduyn Lunel \cite[Theorem 3.2 in Chapter 2]{Hale1993} -- for delay differential equations show that each (in the forward time-direction) non-continuable solution of Eq. \eqref{eq: DDE-canonical} has to leave any closed bounded subset of $C$ in finite time. Now assume that we would have $t_+(\varphi)<\infty$. Then, using the last part, we would see that the orbit $W:=\lbrace w_t\mid 0\leq t<t_{+}(\varphi)\rbrace\subset C$ of solution $w$ is bounded in $C$, and $w$ would clearly not leave the closed bounded set $\overline{W}\subset C$, which is a contradiction. Consequently, $t_{+}(\varphi)=\infty$ and this finishes the proof.
\end{proof}

All the solutions of Eq. \eqref{eq: DDE-canonical} depend continuously on the initial values $\varphi\in C$: Given $\varphi\in C$, $T>0$, and $\varepsilon>0$ there exists a constant $\delta>0$ such that for all $\psi\in C$ with $\|\varphi-\psi\|_{C}<\delta$ and all $0\leq t\leq T$ we have
\begin{equation*}
\|w^{\varphi}(t)-w^{\psi}(t)\|_{\R^2}<\varepsilon.
\end{equation*}
In particular, the segments $w_t^{\varphi}$, $\varphi\in C$ and $0\leq t<\infty$, induce a continuous semiflow on the state space $C$, namely, $F:[0,\infty)\times C\to C$ with
\begin{equation*}
F(t,\varphi):=w_t^{\varphi}
\end{equation*}
for all $t\geq 0$ and $\varphi\in C$.

\section{Existence of wavefront solutions with constant-speed}\label{sec: existence}
Suppose that there exists some $c> 0$ with
\begin{equation}\label{eq: constant-speed condition}
h\, V(c)=c.
\end{equation}
Then, for each $d\in\R$,
\begin{equation}\label{eq: constant-speed solution}
w^{c,d}(t):=\begin{pmatrix}
-c\,t+d\\-c
\end{pmatrix},\qquad t\in\R,
\end{equation}
forms a solution of Eq. \eqref{eq: DDE-canonical}. Indeed, we have
\begin{equation*}
\left(w^{c,d}\right)^{\prime}(t)=\begin{pmatrix}
-c\\0
\end{pmatrix}
\end{equation*}
and
\begin{equation*}
f\left(w^{c,d}_{t}\right)=\begin{pmatrix}
-c\\h^{2}V(-c(t-1)+d-(-ct+d))-hc
\end{pmatrix}=\begin{pmatrix}
-c\\h^{2}V(c)-hc
\end{pmatrix}=\begin{pmatrix}
-c\\0
\end{pmatrix}.
\end{equation*}
With respect to the traffic model described by Eq. \eqref{eq: traffic_model}, such a solution $w^{c, d}$ of Eq. \eqref{eq: DDE} leads to
\begin{equation*}
x_j(t)=z\left(-\frac{1}{h}t-j\right)=-c\left(-\frac{1}{h}-j\right)+d=\frac{c}{h}t+c\,j+d
\end{equation*}
for all $t\in\R$ and all $j\in\mathbb{Z}$, and thus to the dynamical behavior where all cars are uniformly spaced on the road with distance $c$ to the car in front and moving with the same constant velocity $c/h$. For that reason, we should refer to a solutions of type $w^{c,d}$ with $d\in\R$ and $c>0$ satisfying Eq. \eqref{eq: constant-speed condition}  as a {\it wavefront solution with constant-speed}. Furthermore, in view of the translation invariance \eqref{eq: translation_invariance},  it is also appropriate not to distinguish between any two solutions $w^{c,d_1}$ and $w^{c,d_2}$ of Eq. \eqref{eq: DDE-canonical}. Having this in mind, in the following we identify any two solutions $w^{c_1,d_1}$ and $w^{c_2,d_2}$ as one and the same wavefront solution with constant speed as long as $c_1=c_2$.
\begin{remark}
	1. Regardless of the value $h>0$, $c=0$ always satisfies condition \eqref{eq: constant-speed condition}. But that leads to the solution $w^{0, d}$ which was already classified as a non-physical solution of the traffic model given by Eq. \eqref{eq: traffic_model}  (compare Remark \ref{rem: zero speed}).
	
	2. Observe that Eq. \eqref{eq: constant-speed condition} does not have any solutions $c<0$ at all since both $h$ and $V$ are positive by assumptions. Moreover, Eq. \eqref{eq: DDE-canonical} is not only a sufficient but also a necessary condition for the existence of a solution $w$ of Eq. \eqref{eq: DDE-canonical} where the second component is constant.
\end{remark}

Observe that in our discussion above we regarded the parameter $h>0$ as fixed and looked for an appropriate choice of $c>0$ such that Eq. \eqref{eq: constant-speed condition} is satisfied, in order to obtain a wavefront solution with constant-speed. On the other hand, given any $c>0$ with $V(c)\not=0$, the parameter value $h:=c/V(c)>0$ trivially fulfills condition \eqref{eq: constant-speed condition}. In this way, we find a ``branch'' $c\mapsto h(c)=c/V(c)$ of wavefront solutions with constant-speed. However, for the study carried out in this work it is more convenient to parametrize the wavefront solutions with constant-speed by the parameter $h>0$ involved in the right-hand side of Eq. \eqref{eq: DDE}. Therefore, our next goal is to analyze the existence of wavefront solutions with constant-speed in dependence of parameter $h>0$. Our first result in this direction proves that for each $h>0$ the number of such solutions $c>0$ is bounded from above. To be more precisely, the following holds.

\begin{proposition}\label{prop: bound for constant speed}
	For each parameter $h>0$ there are at most two reals $c>0$ satisfying Eq. \eqref{eq: constant-speed condition}.
\end{proposition}
\begin{proof}
	Given $h>0$, consider the function $g:[0,\infty)\ni c\mapsto hV(c)-c\in\R$. By assumptions, $g$ is continuously differentiable and $g(0)=0$. Moreover, $c=0$ is clearly the only zero of $g$ in the interval $[0,d_{S}]$, which collapse to the one-point set $\lbrace 0\rbrace$ in the case $d_{S}=0$.  However, in order to see the claim, it obviously suffices to prove that $g$ has at most two zeros in $(d_{S},\infty)$.  We do that for the two cases $d_S=0$ and $d_S>0$ separately.
	
	1. \textit{Case $d_S=0$.} In this situation, condition (OVF 4) implies that $g^{\prime}(c)=hV^\prime(c)-1$ has at most two zeros in $(0,\infty)$. Indeed, $V^\prime$ is strictly increasing in $(0,b)$ and strictly decreasing in $(b,\infty)$ such that, for fixed $h>0$, the equation $V^\prime(c)=1/h$ clearly has at most two solutions $c>0$. Applying Rolle's theorem, we conclude that $g$ has at most three zeros in $[0,\infty)$. As $g(0)=0$ this proves the assertion in case $d_S=0$. 
	
	2. \textit{Case $d_S>0$.} Under the additional condition $d_S>0$, and so $V(c)=0$ as $0\leq c\leq d_S$,  we have $g(c)=-c$ and $g^\prime(c)=-1$ for all $0\leq c\leq d_S$. Further, by assumption (OVF 4), we see (similarly to the case above) that $g^\prime$ has at most two zeros in $(d_S,\infty)$. Hence, all in all, $g^\prime$ clearly has at most two zero in $(0,\infty)$. Using Rolle's theorem, we conclude that there are at most three different zeros of $g$ in $[0,\infty)$. In view of $g(0)=0$, this finishes the proof.
\end{proof}

But not for each parameter $h>0$ there is some real $c>0$ such that Eq. \eqref{eq: constant-speed condition} is satisfied as we show next.

\begin{proposition}\label{prop: no wavefront}
	Given $h>0$, suppose that $V^{\prime}(c)<\frac{1}{h}$ for all $c>0$. Then there is no real $c>0$ satisfying Eq. \eqref{eq: constant-speed condition}.
\end{proposition}
\begin{proof}
	Set $g_{1}(c):=V(c)$ and $g_{2}(c):=c/h$ as $c\geq 0$. By assumptions, we have $g_{1}(0)=g_{2}(0)$ and $g_{1}^{\prime}(c)< g_{2}^{\prime}(c)$ for all $c>0$. It follows that $g_{2}(c)>g_{1}(c)$ and so $c=h\,g_{2}(c)>h\,g_{1}(c)=h\,V(c)$ as $c>0$. This proves the claim.
\end{proof}
Recall that by assumption (OVF 4) we have $V^{\prime}(b)=\sup_{c\geq 0}V^{\prime}(c)$. Consequently, the last result implies that for each parameter $h>0$ with $V^{\prime}(b)<1/h$ Eq. \eqref{eq: DDE-canonical} does not have any wavefront solutions with constant-speed.

\begin{remark}
	In the situation $d_{S}>0$ the statement of Proposition \ref{prop: no wavefront} is also true under the somewhat weaker assumption $V^{\prime}(c)\leq \tfrac{1}{h}$ as $c>d_{S}$. Indeed, in this situation we have $g_{1}(0)=g_{2}(0)$, $0=g^{\prime}_{1}(c)<g_{2}^{\prime}(c)$ as $0<c\leq d_{S}$, and $g^{\prime}_{1}(c)\leq g_{2}^{\prime}(c)$ as $c>d_{S}$. Hence, $g_{1}(c)<g_{2}(c)$ for all $c>0$, and the assertion follows.
\end{remark}

Provided Eq. \eqref{eq: DDE-canonical} has a wavefront solution $w^{c, 0}$ with constant-speed and some additional conditions are satisfied, our next result ensures the existence of another wavefront solution $w^{\hat{c}, 0}$, $\hat{c}\not=c$, with constant-speed.
\begin{proposition}\label{prop: position of two constant}
	Let parameter $h>0$ be given.
	\begin{itemize}
		\item[(i)] Suppose that there is some $c_{1}>0$ satisfying Eq. \eqref{eq: constant-speed condition} and $V^{\prime}(c_{1})>1/h$. Then there is some $c_{2}>c_{1}$ such that Eq. \eqref{eq: constant-speed condition} holds.
		\item[(ii)] Suppose that $d_{S}>0$ and that there is some $c_{2}>0$ satisfying Eq. \eqref{eq: constant-speed condition} and $V^{\prime}(c_{2})<1/h$. Then there is some $0<c_1<c_{2}$ such that Eq. \eqref{eq: constant-speed condition} holds.
	\end{itemize}
\end{proposition}
\begin{proof}
	1. Consider the continuously differentiable functions $g_{1}$ and $g_{2}$ from the proof of the last statement. Under given conditions of assertion (i), we have $g_{1}(c_{1})=g_{2}(c_{1})$ and $V^{\prime}(c_{1})=g^{\prime}_{1}(c_{1})>g_{2}^{\prime}(c_{1})=1/h$. Thus, $g_{1}(c)-g_{2}(c)>0$ for all $c>c_{1}$ with $c-c_{1}$ sufficiently small. On the other hand, we have $g_{1}(c)-g_{2}(c)\to -\infty$ as $c\to\infty$.
	Thus, the intermediate value theorem yields the existence of some $\xi>c_{1}$ with $g_{1}(\xi)=g_{2}(\xi)$, that is, $h V(\xi)=h\,g_{1}(\xi)=h\,g_{2}(\xi)=\xi$. This proves assertion (i).
	
	2. We consider again the continuously differentiable functions $g_{1}$ and $g_{2}$. By assumptions of assertion (ii), $g_{1}(c_{2})=g_{2}(c_{2})$ and $V^{\prime}(c_{2})=g_{1}^{\prime}(c_{2})<g_{2}^{\prime}(c_{2})=1/h$. Therefore, $(g_{1}-g_{2})(c_{2}-c)>0$ for all sufficiently small $c>0$. On the other hand, we have $(g_{1}-g_{2})(c)=-c/h<0$ for all $0<c\leq d_{S}$. Hence, due to the intermediate value theorem there is some $0<c_{1}<c_{2}$ with $(g_{1}-g_{2})(c_{1})=0$, and this shows the second part of the proposition.
\end{proof}

As an immediate consequence of the last result we get the following corollary.
\begin{corollary}
	Let $V$ satisfy properties (OVF 1) -- (OVF 4) with $d_{S}>0$, and let $h>0$ be given. If there is a unique $c>0$ satisfying Eq. \eqref{eq: constant-speed condition}, then $V^{\prime}(c)=1/h$.
\end{corollary}

We use the criterion of the last result to show that there is at most one pair $(c^{*},h^{*})\in(0,\infty)\times (0,\infty)$ of reals satisfying both
\begin{equation}\label{eq: uni_cond}
h^{\star}\,V(c^{\star})=c^{\star}\qquad\text{and}\qquad h^{\star}\, V^{\prime}(c^{\star})=1
\end{equation}
simultaneously.
\begin{proposition}\label{prop: branch_limit}
	Suppose that there are two pairs $(c_1,h_1), (c_2,h_2)\in (0,\infty)\times(0,\infty)$ satisfying conditions \eqref{eq: uni_cond}. Then $h_1=h_2$ and $c_1=c_2$.
\end{proposition}
\begin{proof}
	1. Contrary to our claim, suppose that there are two pairs $(c_{1},h_{1}),(c_{2},h_{2})$, $(c_1,h_1)\not=(c_2,h_2)$, of positive reals such that we have
	\begin{equation}\label{eq: aux}
	\left\lbrace\begin{aligned}
	h_{1}\,V(c_{1})&= c_{1}\\
	h_{1}\,V^{\prime}(c_{1})&= 1
	\end{aligned}\right.\qquad\text{and}\qquad  \left\lbrace\begin{aligned}
	h_{2}\,V(c_{2})&=c_{2}\\
	h_{2}\,V^{\prime}(c_{2})&=1
	\end{aligned}\right..
	\end{equation}
	As in the situation $c_{1}=c_{2}$ it clearly follows that $h_{1}=h_{2}$ and thus $(h_{1},c_{1})=(h_{2},c_{2})$, it suffices to consider only the case $c_{1}\not=c_{2}$. Moreover, in view of the assumption on $V$, we may assume $d_{S}<c_{1}<c_{2}$.
	
	2. {\it Claim}: $c_{1}<b<c_{2}$, where the constant $b$ is defined in assumption (OVF 4). In order to see this claim, consider the function
	$g:(d_{S},\infty)\to\R$ defined by
	\begin{equation*}
	g(c):=c\,V^\prime(c)-V(c)
	\end{equation*}
	Clearly, $g$ is continuously differentiable and its derivative is given by
	\begin{equation*}
	g^{\prime}(c)=c\,V^{\prime\prime}(c).
	\end{equation*}
	Further, a simple calculation involving assumption \eqref{eq: aux} shows that $g(c_1)=0=g(c_2)$. Therefore, Rolle's theorem implies the existence of some real $c_{1}<\xi<c_{2}$ with $0=g^{\prime}(\xi)=\xi\,V^{\prime\prime}(\xi)$. As $\xi>d_S\geq 0$ it follows first that $V^{\prime\prime}(\xi)=0$ and then, in consideration of condition (OVF 4), $\xi=b$ as claimed.
	
	3. By the last part, we have $g(c_{1})=0$ and $d_S<c_{1}<b$. By applying the mean value theorem, we find some $0<\theta<c_{1}$ with 
	\begin{equation*}
	V(c_{1})=V(0)+V^{\prime}(\theta)\, c_{1}=V^{\prime}(\theta)\,c_1.
	\end{equation*}
	Observe that $0\leq V^{\prime}(\theta)<V^{\prime}(c_{1})$ as $V^{\prime}$ is constant with value zero on interval $[0,d_{s}]$ and strictly increasing in $(d_{s},b)$. Thus, we finally get
	\begin{equation*}
	0=g(c_{1})=c_1\,V^{\prime}(c_1)-V(c_1)=c_1\left[V^{\prime}(c_1)-V^{\prime}(\theta)\right]>0,
	\end{equation*}
	a contradiction to the existence of two pairs $(c_{1},h_{1}),(c_{2},h_{2})$, $(c_1,h_1)\not=(c_2,h_2)$, of positive reals satisfying assumption \eqref{eq: aux}. This finishes the proof.
\end{proof}

As the last preparatory step towards the main result of this section we show that there indeed exists a pair $(c^\star,h^\star)$ of positive reals satisfying \eqref{eq: uni_cond}.
\begin{proposition}\label{prop: degenerate_constant}
	Under the given assumptions on $V$, there is a uniquely determined pair $(c^{\star},h^{\star})\in(0,\infty)\times(0,\infty)$ with the property \eqref{eq: uni_cond}.
\end{proposition}
\begin{proof}
	1. First observe that it suffices to prove the existence of such a pair $(c^{\star},h^{\star})$ of reals since the uniqueness part of the assertion would immediately follow from Proposition \ref{prop: branch_limit}. In order to see the existence, fix constant $d_{S}<c_{1}<b$ and set $\tilde{h}:=c_{1}/V(c_{1})>0$. Clearly, the reals $\tilde{h}$ and $c_{1}$ satisfy the constant-speed condition given by Eq. \eqref{eq: constant-speed condition}. Moreover, we claim that $V^{\prime}(c_{1})>1/\tilde{h}$ such that Proposition \ref{prop: position of two constant} implies the existence of another constant $c_{2}>c_{1}$ with $\tilde{h}\,V(c_{2})=c_{2}$. In order to see the claim, first recall from the assumptions (OVF 1) -- (OVF 4) that $V^{\prime}$ is positive and strongly monotonically increasing on $(d_{s},b)$. Now, write $V(c_{1})$ as
	\begin{equation*}
	V(c_{1})=V(0)+V^{\prime}(\xi)c_{1}=V^{\prime}(\xi)\,c_{1}
	\end{equation*}
	with some $0<\xi<c_{1}$. As $V(c_{1})>0$ we clearly have $V^{\prime}(\xi)>0$ and so $\xi> d_{S}$. It follows that
	\begin{equation*}
	V^{\prime}(c_{1})\,\tilde{h}=\frac{V^{\prime}(c_{1})\,c_{1}}{V(c_{1})}=
	\frac{V^{\prime}(c_{1})\,c_{1}}{V^{\prime}(\xi)\,c_{1}}>1,
	\end{equation*}
	and so $V^{\prime}(c_{1})>1/\tilde{h}$. Thus, there is in fact some $c_{2}>c_{1}$ with $\tilde{h}\,V(c_{2})=c_{2}$.
	
	2. Next, note that we have $V^{\prime}(c_{2})\leq 1/\tilde{h}$. Indeed, otherwise an application of Proposition \ref{prop: position of two constant} would show the existence of some further constant $c_{3}>c_{2}$ with $\tilde{h}\,V^{\prime}(c_{3})=c_{3}$, in contradiction to Proposition \ref{prop: bound for constant speed}.
	
	In the case $V^{\prime}(c_{2})=1/\tilde{h}$, the assertion obviously follows with reals $h^{\star}=\tilde{h}$ and $c^{\star}=c_{2}$. Therefore, assume that $V^{\prime}(c_{2})<1/\tilde{h}$ in the following, and let function $g:(d_{s},\infty)\to(0,\infty)$ be given by
	\begin{equation*}
	g(c):=\frac{c\,V^{\prime}(c)}{V(c)}.
	\end{equation*}
	Clearly, $g$ is continuous, and using the first part, we get
	\begin{equation*}
	g(c_{2})=\frac{c_{2}\,V^{\prime}(c_{2})}{V(c_{2})}=\tilde{h}\,V^{\prime}(c_{2})<1<
	\tilde{h}\,V^{\prime}(c_{1})=\frac{c_{1}\,V^{\prime}(c_{1})}{V(c_{1})}=g(c_{1}).
	\end{equation*}
	It follows that there exists $c_{1}<c^{\star}<c_{2}$ with $g(c^{\star})=1$. Setting $h^{\star}:=c^{\star}/V(c^{\star})>0$, we obviously get a pair $(c^{\star},h^{\star})$ of positive reals satisfying condition \eqref{eq: uni_cond}.
\end{proof}

Now, we are in the position to state and prove our main result of this section.
\begin{theorem}\label{thm: two_branches}
	Let $V$ satisfy the standing assumptions (OVF 1) -- (OVF 4). Then, apart from the pair $(c^{\star},h^{\star})\in (0,\infty)\times (0,\infty)$ obtained from Proposition \ref{prop: degenerate_constant}, all other solutions $(c,h)\in (0,\infty)\times(0,\infty)$ of Eq. \eqref{eq: constant-speed condition} are given by two branches $h\mapsto (c(h),h)$ defined by two functions
	\begin{equation}
	c_{1}:(h^{\star},\hat{h})\to (d_{S},c^{\star})\qquad\text{and}\qquad
	c_{2}:(h^{\star},\infty)\to (c^{\star},\infty)
	\end{equation}
	where $h^\star<\hat{h}\leq \infty$. Both, $c_{1}$ and $c_{2}$, are continuously differentiable with
	\begin{equation*}
	c^{\prime}_{i}(h)=\frac{V(c_{i}(h))}{1-h\,V^{\prime}(c_i(h))}\,,
	\end{equation*}
	where
	\begin{equation*}
	h\,V^{\prime}(c_{1}(h))>1
	\end{equation*}
	along the first branch and, conversely,
	\begin{equation*}
	h\,V^{\prime}(c_{2}(h))<1
	\end{equation*}
	along the second branch.
\end{theorem}
\begin{proof}
	1. To begin with, observe that in the case of $0\leq c\leq d_{S}$ we have $V(c)=0$ and thus in this situation there clearly is no real $h>0$ such that Eq. \eqref{eq: constant-speed condition} holds. Therefore, it is sufficient to study only the case $c>d_{S}$ below.
	
	2. \textit{Definition of $c_{1}$.} Given $d_{S}<c<c^{\star}$, set $h:=c/V(c)>0$. Of course, the pair $(c,h)$ fulfills Eq. \eqref{eq: constant-speed condition}. Furthermore, we claim that $h\,V^{\prime}(c)>1$. In order to see this, it suffices to consider $b\leq c<c^{\star}$ as in the situation $d_{S}<c<b$ the assertion immediately follows from the proof of Proposition \ref{prop: degenerate_constant}. In particular, there is a pair $(\tilde{c},\tilde{h})$ with $d_S<\tilde{c}<b$ satisfying both $\tilde{h}\,V(\tilde{c})=\tilde{c}$ and $\tilde{h}\,V^{\prime}(\tilde{c})>1$. Now, assume that $b\leq c< c^{\star}$ but $h\,V^{\prime}(c)\leq 1$. Then, in consideration of $c\not= c^{\star}$ and of Proposition \ref{prop: degenerate_constant}, it follows that $h\,V^{\prime}(c)<1$. Moreover, for the continuous function $g$ used in the proof of Proposition \ref{prop: degenerate_constant} we have $g(c)<1<g(\tilde{c})$. For this reason, the intermediate value theorem implies the existence of some $\tilde{c}<c_{0}<c<c^{\star}$ with
	\begin{equation*}
	1=g(c_{0})=\frac{c_{0}\,V^{\prime}(c_{0})}{V(c_{0})}.
	\end{equation*}
	Setting $h_{0}:=c_{0}/V(c_{0})>0$, we see that the pair $(c_{0},h_{0})$ of positive reals satisfies $h_{0}\,V(c_{0})=c_{0}$ and $h_{0}\,V^{\prime}(c_{0})=1$. But this contradicts Proposition \ref{prop: branch_limit} as $c_{0}<c^{\star}$. Hence, $h\,V^{\prime}(c)>1$ as claimed.
	
	Consider now the map $H_{1}:(d_{S},c^{\star})\ni c\mapsto c/V(c)=:h\in(0,\infty)$. $H_{1}$ is clearly continuously differentiable, and for the derivative we get
	\begin{equation*}
	H^{\prime}_{1}(c)=\frac{1}{V(c)}-\frac{c\,V^{\prime}(c)}{V^2(c)}=\frac{1}{V(c)}-
	\frac{h\,V^{\prime}(c)}{V(c)}=\frac{1-h\,V^{\prime}(c)}{V(c)}<0,
	\end{equation*}
	that is, $H_{1}$ is strictly decreasing. Thus, after setting $\hat{h}:=\lim_{c\searrow d_{S}}H_{1}(c)$, we find a continuously differentiable map $c_{1}:(h^{\star},\hat{h})\to (d_{S},c^{\star})$ such that $(H_{1}\circ c_{1})(h)=h$ for all $h\in (h^{\star},\hat{h})$ and $(c_{1}\circ H_{1})(c)=c$ for all $c\in(d_{S},c^{\star})$. In particular,
	\begin{equation*}
	c_{1}^{\prime}(h)=\frac{1}{H^{\prime}_{1}(c_{1}(h))}=\frac{V(c_{1}(h))}
	{1-h\,V^{\prime}(c_{1}(h))}<0.
	\end{equation*}
	
	3. Fix some $h^{\star}<h_{1}<\hat{h}$. Then the last part implies that $(c_{1}(h_{1}),h_{1})$ satisfies Eq. \eqref{eq: constant-speed condition} and $h_{1}\,V^{\prime}(c_{1}(h_{1}))>1$. Furthermore, by Proposition \ref{prop: position of two constant} there is an additional constant $\tilde{c}_{1}>c_{1}(h_{1})$ such that $h_{1}\,V(\tilde{c}_{1})=\tilde{c}_{1}$ holds. We claim that $\tilde{c}_{1}>c^{\star}$ and $h_{1}\,V^{\prime}(\tilde{c}_{1})<1$. To see this, first observe that from the last part it follows easily that $\tilde{c}_{1}\geq c^{\star}$. Next, we surely have $h_{1}\,V^{\prime}(\tilde{c}_{1})\leq 1$ since otherwise Proposition \ref{prop: position of two constant} would imply the existence of a third constant $c_{3}>\tilde{c}_{1}$ with $h_{1}\,V(c_{3})=c_{3}$, in contradiction to Proposition \ref{prop: bound for constant speed}. But if $h_{1}\,V^{\prime}(\tilde{c}_{1})\leq 1$ then, in consideration of $h_{1}\not=h^{\star}$ and Proposition \ref{prop: degenerate_constant}, necessarily $h_{1}\,V^{\prime}(\tilde{c}_{1})<1$ holds. As finally the assumption $\tilde{c}_{1}=c^{\star}$ also results in a contradiction, namely,
	\begin{equation*}
	h_{1}=\frac{\tilde{c}_{1}}{V(\tilde{c}_{1})}=\frac{c^{\star}}{V(c^{\star})}=h^{\star},
	\end{equation*}
	we see that $\tilde{c}_{1}>c^{\star}$ and $h_{1}\,V^{\prime}(\tilde{c}_{1})<1$ as claimed.
	
	4. \textit{Definition of $c_{2}$.} Consider any $c^{\star}<c<\infty$, and set again $h:=c/V(c)>0$. We assert that we have $h\,V^{\prime}(c)<1$. In order to show this, recall from the last part that in the case $c=\tilde{c}_1$ we clearly have $h\,V^{\prime}(c)<1$. Otherwise, $c\not=\tilde{c}_1$ and we assume, contrary to the assertion, that $h\,V^{\prime}(c)\geq 1$. Then either $h\,V^{\prime}(c)=1$ or $h\,V^{\prime}(c)>1$. However, both situations lead to a contradiction. First observe that due to Proposition \ref{prop: degenerate_constant} the case $h\,V^{\prime}(c)=1$ results in $c=c^{\star}$, and thus indeed in a contradiction to $c>c^{\star}$. Next, consider the situation $h\,V^{\prime}(c)>1$. For the map $g$ from the proof of Proposition \ref{prop: degenerate_constant} we get
	\begin{equation*}
	g(\tilde{c}_{1})<1<g(c)
	\end{equation*}
	with $\tilde{c}_{1}>c^{\star}$ from the second part. Hence, due to the intermediate value theorem there is some $c_{0}\in(\tilde{c}_{1},c)$ or $c_{0}\in(c,\tilde{c}_{1})$ such that $g(c_{0})=1$ holds. But then the pair $(c_{0},h_{0})$ with $h_{0}:=c_{0}/V(c_{0})$ satisfy $h_{0}\,V(c_{0})=c_{0}$ and $h_{0}\,V^{\prime}(c_{0})=1$, in contradiction to Proposition \ref{prop: degenerate_constant} as $c_{0}\not=c^{\star}$. For this reason, we have $h\,V^{\prime}(c)<1$ as claimed.
	
	Let now the map $H_{2}:(c^{\star},\infty)\to (0,\infty)$ be defined by $H_{2}(c)=c/V(c)$. Then $H_{2}$ is continuously differentiable with
	\begin{equation*}
	H_{2}^{\prime}(c)=\frac{1-h\,V^{\prime}(c)}{V(c)}>0
	\end{equation*}
	for all $c^\star <c<\infty$. Thus, $H_{2}$ is strictly increasing, and we have 
	\begin{equation*}
	\lim_{c\searrow c^{\star}}H_{2}(c)=c^{\star}/V(c^{\star})=h^{\star}
	\end{equation*}
	and $H_{2}(c)\to\infty$ as $c\to\infty$. Consequently, there is a continuously differentiable map $c_{2}:(h^{\star},\infty)\to (c^{\star},\infty)$ such that $(H_{2}\circ c_{2})(h)=h$ for all $h\in(h^{\star},\infty)$ and $(c_{2}\circ H_{2})(c)=c$ for all $c\in(c^{\star},\infty)$. Finally, we get
	\begin{equation*}
	c_{2}^{\prime}(h)=\frac{1}{H_{2}^{\prime}(c_{2}(h))}=\frac{V(c_{2}(h))}{1-h\,
		V^{\prime}(c_{2}(h))}>0,
	\end{equation*}
	which closes the proof.
\end{proof}

\begin{remark}
	1. Observe that in the case $d_{S}>0$ we have $\hat{h}=\infty$. In all others situation, the value of $\hat{h}$ depends on the value of $V^{\prime}(0)$ as a simple argument shows.
	
	2. Of course, the existence of the two functions $c_{1}$ and $c_{2}$ can also be obtained by a straightforward application of the implicit function theorem. Indeed, if we assume that there are $\tilde{h},\tilde{c}>0$ with $\tilde{h}\,V(\tilde{c})=\tilde{h}$ and $V^{\prime}(\tilde{c})\not=1/\tilde{h}$ then the continuously differentiable function $g:(0,\infty)\times(0,\infty)\to\R$ given by $g(c,h):=h\,V(c)-c$ satisfies
	\begin{equation*}
	g(\tilde{c},\tilde{h})=0\qquad\text{and}\qquad\frac{\partial g}{\partial c}(\tilde{c},\tilde{h})=\tilde{h}\,V^{\prime}(\tilde{c})-1\not=0.
	\end{equation*}
	But, however, first you have to find some reals $\tilde{c},\tilde{h}>0$ satisfying $\tilde{h}\,V(\tilde{c})=\tilde{c}$ and $\tilde{h}\,V^{\prime}(\tilde{c})\not=1$ simultaneously, and that should not make the proof substantially shorter than the one discussed above.
\end{remark}

\section{The linearization along a wavefront solution with constant-speed and its spectral properties }\label{sec: linearization}
After finding all wavefront solutions with constant speed of Eq. \eqref{eq: DDE-canonical}, we are now interested in the stability properties of these solutions.  In order to address this issue  by the principles of linearized stability and instability in the next section, we first analyze the linearization of Eq. \eqref{eq: DDE-canonical} along such a solution and study its spectral properties.

Throughout this section, let $w^{c,d}:\R\to\R^2$, $c>0$ and $d\in\R$, be a fixed wavefront solution with constant-speed  of Eq. \eqref{eq: DDE-canonical}.  Before linearizing Eq. \eqref{eq: DDE-canonical} along $w^{c,d}$, it is convenient to translate first the solution $w^{c,d}$ to the origin as follows: Let
\begin{equation*}
w(t):=w^{c,d}(t)+v(t)
\end{equation*}
denote a solution of Eq. \eqref{eq: DDE-canonical} which incorporates a small perturbation $v:I\to\R^2$, $I\subset\R$ an interval, of the wavefront solution $w^{c,d}$ with constant-speed. Then, by inserting this ansatz for a solution $w$ of Eq. \eqref{eq: DDE-canonical} into the differential equation, we get
\begin{equation*}
(w^{c,d})^{\prime}(t)+v^{\prime}(t)=f(w_{t}^{c,d}+v_{t})
\end{equation*}
and so
\begin{equation}\label{eq: DDE_2}
v^{\prime}(t)=\tilde{f}(v_{t})
\end{equation}
with the map $\tilde{f}:C\to\R^2$ defined by
\begin{equation*}
\tilde{f}(v_{t}):=f\left(v_{t}+w_{t}^{c,d}\right)-f\left(w_{t}^{c,d}\right).
\end{equation*}
The map $\tilde{f}$ has, of course, the same smoothness and compactness properties as $f$. In particular, for each $\varphi\in C$, Eq. \eqref{eq: DDE_2} has a uniquely determined solution $v^{\varphi}:[-1,\infty)\to\R^{2}$ with $v^{\varphi}_0=\varphi$. The segments $v_t^{\varphi}$, $\varphi\in C$ and $t\geq 0$, form a continuous semiflow $\tilde{F}:[0,\infty)\times C\to C$ with $\tilde{F}(t,\varphi):=v_t^{\varphi}$. This semiflow $\tilde{F}$ is obviously closely related to the semiflow $F$ induced by the solution of Eq. \eqref{eq: DDE-canonical}. Indeed, we have
\begin{equation*}
\tilde{F}(t,\varphi)+w^{c,d}_t=v^{\varphi}_t+w^{c, d}_t=F(t,\varphi+w^{c,d}_0)
\end{equation*}
for all $t\geq 0$ and all $\varphi\in C$.

The solution $w^{c,d}$ of Eq. \eqref{eq: DDE-canonical} is now represented by the zero solution of Eq. \eqref{eq: DDE_2}, and it is unstable / stable / (locally) asymptotically stable if and only if the zero solution of Eq. \eqref{eq: DDE_2}, that is, the stationary point $\varphi_0:=0\in C$ of the semiflow $\tilde{F}$, is unstable / stable / (locally) asymptotically stable.

Now, we linearize Eq. \eqref{eq: DDE-canonical} along the wavefront solution $w^{c,d}$, or equivalently, we linearize Eq. \eqref{eq: DDE_2} along the zero solution. For this reason, we calculate the derivative of the map $f$ defining the right-hand side of Eq. \eqref{eq: DDE-canonical}. At each $\varphi\in C$ it forms a bounded linear operator $Df(\varphi)\in\mathcal{L}(C,\R^{2})$ whose action to some $\psi\in C$ is given by
\begin{equation*}
\begin{aligned}
Df(\varphi)\psi&=D(G\circ(\mathrm{ev}_{0}\times\mathrm{ev}_{-1}))(\varphi)(\psi)\\
&=\big(DG((\mathrm{ev}_{0}\times\mathrm{ev}_{-1})(\varphi))\circ
D(\mathrm{ev}_{0}\times\mathrm{ev}_{-1})(\varphi)\big)(\psi)\\
&=\big(DG(\varphi(0),\varphi(-1))\big((\mathrm{ev}_{0}\times\mathrm{ev}_{-1})(\psi)\big)\\
&=\begin{pmatrix}
\psi_{2}(0)\\ \\h^{2}\,V^{\prime}(\varphi_{1}(-1)-\varphi_{1}(0))\left[\psi_{1}(-1)-\psi_{1}(0)\right]+h\,\psi_{2}(0)
\end{pmatrix}.
\end{aligned}
\end{equation*}
Hence, along the zero solution of Eq. \eqref{eq: DDE_2} the associated linear delay equation reads
\begin{equation}\label{eq: linear DDE}
y^{\prime}(t)=L\,y_{t}
\end{equation}
with $L\in\mathcal{L}(C,\R^{2})$ defined by
\begin{equation*}
\begin{aligned}
L\,\varphi&=D\tilde{f}(0)\,\varphi\\
&=Df(w^{c,d}_{t})\,\varphi\\
&=\begin{pmatrix}
\varphi_{2}(t)\\ \\ h^{2}V^{\prime}(c)\left[\varphi_{1}(t-1)-\varphi_{1}(t)\right]+h\,\varphi_{2}(t)
\end{pmatrix}.
\end{aligned}
\end{equation*}
Of course, Eq. \eqref{eq: linear DDE} is equivalent to the ``linearization''
\begin{equation*}
u^{\prime\prime}(t)=h^2\,V^{\prime}(c)\left(u(t-1)-u(t)\right)+h\,u^{\prime}(t)
\end{equation*}
of the scalar delay differential equation \eqref{eq: DDE} along the solution $z(t)=-c\,t+d$.
However, we stay in the more functional analytical setting of retarded functional differential equations and use a consequence of the Riesz representation theorem, which enables -- compare, for instance, Diekmann et al. \cite[Theorem 1.1 in Chapter I]{Diekmann1995} -- to represent the action of the operator $L$ to some $\varphi\in C$ by a Riemann-Stieltjes integral
\begin{equation*}
L\,\varphi=\int_{0}^{1}d\zeta(\tau)\varphi(-\tau)
\end{equation*}
involving a uniquely determined normalized function $\zeta:[0,1]\to\R^{2\times 2}$ of bounded variation. In this context, the normalization conditions means that $\zeta$ should satisfy $\zeta(0)=0$ and be continuous from the right on $(0,1)$, that is, $\zeta(\tau)=\zeta(\tau+)$ for all $0<\tau<1$.  A straightforward calculation shows that
\begin{equation*}
\zeta(\tau):=\begin{cases}
\begin{pmatrix}
0&0\\0&0
\end{pmatrix},&\tau=0,\\&\\
\begin{pmatrix}
0&1\\-h^{2}\,V^{\prime}(c)&h
\end{pmatrix},&0<\tau<1,\\ & \\
\begin{pmatrix}
0&1\\
0&h
\end{pmatrix},&\tau=1.
\end{cases}
\end{equation*}
For each $\varphi\in C$, Eq. \eqref{eq: linear DDE} has a uniquely determined solution $y^{\varphi}:[-1,\infty)\to\R^{2}$. The equations
\begin{equation*}
T(t)\,\varphi=y^{\varphi}_{t},
\end{equation*}
with $\varphi\in C$ and $t\geq 0$ define a strongly continuous semigroup $T=\lbrace T(t)\rbrace_{t\geq 0}$ of bounded linear operators $T(t):C\to C$. The infinitesimal generator $G:\mathcal{D}(G)\to C$ is given by
\begin{equation*}
G\,\varphi=\varphi^{\prime}
\end{equation*}
where
\begin{equation*}
\mathcal{D}(G):=\left\lbrace\varphi\in C\mid \varphi^{\prime}\in C, \varphi^{\prime}(0)=L\varphi\right\rbrace.
\end{equation*}

The associated characteristic matrix reads
\begin{equation*}
\triangle(\lambda):=\lambda\, E_{2}-\int_{0}^{1}e^{-\lambda\tau}\,d\zeta(\tau)
=\begin{pmatrix}
\lambda &-1\\
h^{2}\,V^{\prime}(c)[1-e^{-\lambda}]&\lambda-h
\end{pmatrix}
\end{equation*}
such that for the characteristic equation of Eq. \eqref{eq: linear DDE} we get
\begin{equation}\label{eq: char}
\det\triangle(\lambda)=\lambda^{2}-h\lambda+h^{2}\,V^{\prime}(c)\left[1-e^{-\lambda}\right]=0.
\end{equation}
The (countably infinitely many) solutions of this algebraic equation coincide with the spectrum $\sigma(G)\subset \mathbb{C}$ of $G$. The last consists only of eigenvalues with finite rank, that is, the associated generalized eigenspaces are finite dimensional, and for each real $\beta>0$ the spectral subset $\lbrace \lambda\in\sigma(G)\mid\mathrm{Re}\,(\lambda)>\beta\rbrace$ is either empty or finite. Writing $\sigma_{u}(G)$, $\sigma_{c}(G)$, and $\sigma_{s}(G)$ for the spectral subsets of $\sigma(G)$ consisting of eigenvalues with negative, zero, and positive real parts, respectively, we get the splitting 
\begin{equation*}
\sigma(G)=\sigma_u(G)\cup\sigma_c(G)\cup\sigma_s(G)
\end{equation*}
of $\sigma(G)$. Furthermore, let $C_{u}$, $C_{c}$ and $C_{s}$ denote the associated realified generalized eigenspaces, which are called the unstable, the center and the stable space of $G$. Then the Banach space $C$ decomposes to
\begin{equation*}\label{eq: decomposition}
C=C_{u}\oplus C_{c}\oplus C_{s},
\end{equation*}
with the two $C_u$, $C_c$ obviously finite and the one $C_s$ in general infinite dimensional subspaces.

Now, it is apparent that for any $h>0$ and any $V^{\prime}(c)> 0$, we always have $\lambda_0=0\in\sigma(G)$; that is, the linearization along any wavefront solution with constant-speed has a zero eigenvalue. In the case $h\,V^{\prime}(c)\not=1$, the eigenvalue $\lambda_{0}=0$ is clearly simple, whereas in the case $h\,V^{\prime}(c)=1$ it has the algebraic multiplicity two or three. The permanent occurrence of the zero eigenvalue is caused by the translation symmetry of $f$. To be more precisely, the direction of the translation invariance of $f$ is always an eigendirection of the zero eigenvalue as we shall see next.
\begin{proposition}\label{prop: simple_zero}
	Given a wavefront solution $w^{c, d}$ with constant-speed of Eq. \eqref{eq: DDE-canonical}, let $N\subset C_c$ denote the realified generalized eigenspace of $G$ associated with the eigenvalue $\lambda_{0}=0$. Then $\R\,\hat{e}_1\subseteq N$ with $\hat{e}_1\in C$ defined by Eq. \eqref{eq: center-direction}.
\end{proposition}
\begin{proof}
	Of course, $\hat{e}_1$ is $C^1$-smooth and $(\hat{e}_1)^\prime=0\in C$. Further, an easy computation shows that $L\hat{e}_1=0\in\R^2$. Hence, we clearly have $\hat{e}_1\in\mathcal{D}(G)$. Moreover, as
	$G\hat{e}_1=(\hat{e}_1)^\prime=0=\lambda_0 \hat{e}_1$ it follows that $\hat{e}_1\in N$ and so $\mathbb{R}\hat{e}_1\subseteq N$ as claimed.
\end{proof}

Observe that, apart from $\lambda_0=0$,  all other elements of $\sigma(G)$ may in general not be calculated explicitly. However, we are only interested in the location of  the eigenvalues of $G$, and thus of the roots of Eq. \eqref{eq: char}, in the complex plane relative to the imaginary axis. For this reason, consider the function
\begin{equation}\label{eq: char_general}
\chi(\lambda):=\lambda^{2}+\alpha\,\lambda+\beta[1-e^{-\lambda}]
\end{equation}
where $\alpha,\beta\in\R$. By identifying $\alpha=-h$ and $\beta=h^{2}\,V^{\prime}(c)$, the function $\chi$ clearly coincides with the left-hand side of Eq. \eqref{eq: char}.  Therefore,  we may use the following result about the number of unstable characteristic roots of $\chi$ for parameter values $\alpha< 0<\beta$, in order to analyze the stability of wavefront solutions with constant-speed.
\begin{figure}[ht]
	\includegraphics[width=10cm]{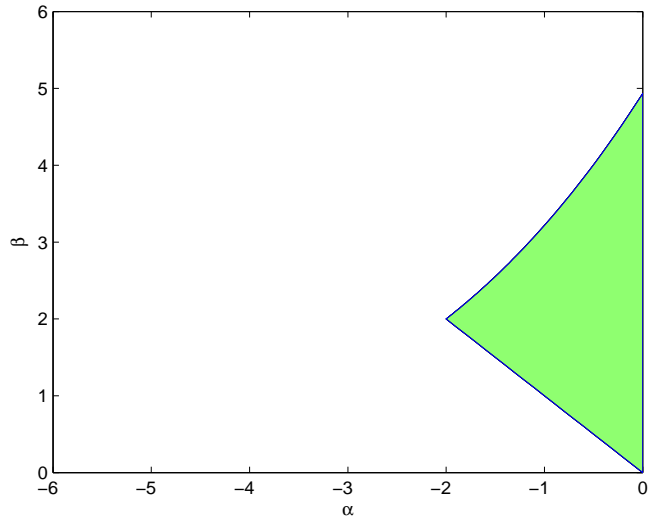}
	\vspace*{-0.5cm}
	\caption{The region $S$ from Proposition \ref{prop: region_S}}\label{fig: S}
\end{figure}
\begin{proposition}\label{prop: region_S}
	Let $S\subset (-\infty,0]\times[0,\infty)$ denote the region, which is bounded by the two lines
	\begin{align*}
	G_{0}&:=\big\lbrace(\alpha,\beta)\in (-\infty,0]\times[0,\infty)\mid \alpha=0,\beta\leq \pi^2/2\big\rbrace,\\
	G_{1}&:=\big\lbrace(\alpha,\beta)\in(-\infty,0]\times[0,\infty)\mid\alpha\geq -2, \alpha+\beta=0\big\rbrace,
	\intertext{and the parametrized curve}
	C_{1}&:=\bigg\lbrace\bigg(-\frac{\nu}{\tan(\nu/2)},\frac{\nu^2}
	{\tan^{2}(\nu/2)\cdot(1+\cos(\nu))}\bigg)\Big\vert 0<\nu<\pi\bigg\rbrace.
	\end{align*}
	
	\begin{itemize}
		\item[(i)] If $(\alpha,\beta)\in S$ then, apart from the simple root $\lambda_{0}=0$, all other roots $\lambda\in\mathbb{C}$ of $\chi$ defined by Eq. \eqref{eq: char_general} are located in the left open half-plane of $\mathbb{C}$.
		\item[(ii)] If $(\alpha,\beta)\in C_{1}$ then, in addition to the simple root $\lambda_{0}=0$, there is a pair $\pm i\omega$, $\omega>0$, of simple pure imaginary roots of $\chi$ given by Eq. \eqref{eq: char_general}. All other roots $\lambda\in \mathbb{C}$ of $\chi$ satisfy $\mathrm{Re}(\lambda)<0$.
		\item[(iii)] If $(\alpha,\beta)\in ((-\infty,0]\times[0,\infty))\backslash\overline{S}$ then there exists at least one root $\lambda\in\mathbb{C}$ of $\chi$ from Eq. \eqref{eq: char_general} with $\mathrm{Re}(\lambda)>0$.
	\end{itemize}
\end{proposition}
\begin{proof}
	Introducing $D:\mathbb{C}\times \R\times\R\to \mathbb{C}$ by
	\begin{equation*}
	D(\lambda,a,b):=\begin{cases}
	\lambda-a-b\frac{1-e^{\lambda}}{\lambda},&\lambda\not=0,\\
	-a-b,&\lambda=0,
	\end{cases}
	\end{equation*}
	we see that
	\begin{equation*}
	\chi(\lambda)=\lambda\cdot D(\lambda,-\alpha,-\beta)
	\end{equation*}
	and so
	\begin{equation*}
	\chi(\lambda)=0\qquad\Leftrightarrow\qquad \left[\,\lambda=0\quad \vee\quad D(\lambda,-\alpha,-\beta)=0\,\right].
	\end{equation*}
	Therefore, it suffices to determine, in dependence on $\alpha$ and $\beta$, the location of the roots $\lambda$ of $D$ in the complex plane. But such an analysis can be found, for instance, in Insperger and St\'ep\'an \cite[Chapter 2.1.2]{Insperger2011}, and this completes the proof.
\end{proof}

\section{Local stability analysis of wavefront solutions with constant-speed}\label{sec: stability}
With the preparatory work of the last section, we are now in the position to analyze the local stability properties of the wavefront solutions $w^{c, d}$ with constant-speed of Eq. \eqref{eq: DDE-canonical}. But before doing so, recall that the zero solution $v^{0}:\R\ni t\mapsto 0\in\R^2$ of Eq. \eqref{eq: DDE_2} is called {\it stable} if and only if for each $\delta>0$ there is some constant $\varepsilon>0$ such that for each $\varphi\in C$ with $\|\varphi\|_C<\delta$ we have $\|v^{\varphi}_t\|_C<\varepsilon$ for all $t\geq 0$. Otherwise, we call the zero solution $v^{0}$ of Eq. \eqref{eq: DDE_2} {\it unstable}. If $v^{0}$ is stable and, additionally, we find some constant $\varepsilon_a>0$ such that for all $\varphi\in C$ with $\|\varphi\|_C<\varepsilon_a$ we have
\begin{equation*}
v^{\varphi}_t\to 0\in C\qquad\text{as}\quad t\to\infty,
\end{equation*}
then the zero solution $v^{0}$ of Eq. \eqref{eq: DDE_2} is called {\it locally asymptotically stable}.

In terms of a wavefront solution $w^{c, d}$ with constant-speed of Eq. \eqref{eq: DDE-canonical}, the above definitions mean the following: The solution $w^{c, d}$ is stable whenever for every $\varepsilon>0$ there exists some $\varepsilon>0$ such that $\|w^{c, d}_0-\varphi\|_C<\delta$ for some $\varphi\in C$ guarantees that
\begin{equation*}
\|w_t^{c, d}-w^{\varphi}_t\|_C<\varepsilon
\end{equation*}
for all $t\geq 0$. Otherwise, the solution $w^{c, d}$ is unstable, that is, there exists some $\varepsilon>0$ such that any neighborhood of $w^{c, d}_0$ in $C$ contains an initial value $\varphi\in C$ with $\|w^{c, d}(t)-w^{\varphi}(t)\|_{\R^2}>\epsilon$ for some $t\geq 0$. Finally, $w^{c, d}$ is locally asymptotically stable when it is stable and, in addition, there is some $\varepsilon>0$ with the property that $\|w^{c, d}_0-\varphi\|_C<\varepsilon_a$ for any $\varphi\in C$ guarantees
\begin{equation*}
w_t^{\varphi}\to w_t^{c, d}\qquad\text{as}\quad t\to\infty.
\end{equation*}

We begin our (local) stability analysis of the wavefront solutions with constant-speed with a result which is hardly surprising in consideration of the translation invariance \eqref{eq: translation_invariance} of $f$.
\begin{proposition}
	Under given assumptions, Eq. \eqref{eq: DDE-canonical} does not have any wavefront solutions with constant-speed which are (locally) asymptotically stable.    
\end{proposition} 
\begin{proof}
	Given $c>0$, $d\in\R$ and a wavefront solution $w^{c, d}$ with constant-speed of Eq. \eqref{eq: DDE-canonical},  consider for any $\varepsilon_a>0$ the constant function
	\begin{equation*}
	v^{\varepsilon_a}:\R\ni t\mapsto \begin{pmatrix}
	\frac{\varepsilon_a}{2} \\0
	\end{pmatrix}\in\R^2.
	\end{equation*}
	Then, by the translation invariance of $f$, $w=w^{c,d}+v^{\varepsilon_a}$
	trivially forms a solution of Eq. \eqref{eq: DDE-canonical} and so $v^{\varepsilon_a}$ a solution of Eq. \eqref{eq: DDE_2}. Indeed, for all $t\in\R$ we have $(v^{\varepsilon_a})^{\prime}(t)=0$ and
	\begin{equation*}
	\tilde{f}(v^{\varepsilon_a}_t)=f(w^{c, d}_t+v_t^{\varepsilon_a})-f(w^{c, d}_t)=f(w_t^{c, d}+\tfrac{\varepsilon_a}{2}\,\hat{e}_1)-f(w^{c, d}_t)=f(w^{c, d}_t)-f(w^{c, d}_t)=0
	\end{equation*}
	with function $\hat{e}_1\in C$ defined by Eq. \eqref{eq: center-direction}. Now $v^{\varepsilon_a}_t=\tfrac{\varepsilon_a}{2}\hat{e}_1$ as $t\geq 0$. In particular, $\|v^{\varepsilon_a}_0\|_C<\varepsilon_a$. But $v^{\varepsilon_a}_t$ clearly does not converge to $0\in C$ as $t\to\infty$. This proves the assertion.
\end{proof}

But it is even more sobering when we assume that the wavefront solutions $w^{c,d}$ with constant-speed lies on the second branch from Theorem \ref{thm: two_branches}:
\begin{theorem}\label{thm: stability of second branch}
	Given $V$ satisfying the standing hypotheses (OVF 1) -- (OVF 4), let $w^{c_2(h), d}$ with $h^{\star}<h<\infty$ and $d\in\R$ denote a wavefront solution with constant-speed of Eq. \eqref{eq: DDE-canonical} belonging to the second branch from Theorem \ref{thm: two_branches}. Then $w^{c_2(h), d}$ is unstable.
\end{theorem}
\begin{proof}
	Under given assumptions, $h\,V(c_{2}(h))=c_{2}(h)$ and $h\,V^{\prime}(c_{2}(h))<1$ due to Theorem \ref{thm: two_branches}. Multiplying the last inequality with $h$, we see $h^{2}\,V^{\prime}(c_{2}(h))<h$. Hence, for $\beta=h^{2}\,V^{\prime}(c_{2}(h))$ and $\alpha=-h$ it trivially follows $\beta<-\alpha$, and therefore $(\alpha,\beta)\in ((-\infty,0]\times[0,\infty))\backslash \overline{S}$ with the region $S$ from Proposition \ref{prop: region_S}. Consequently, by assertion (iii) of Proposition \ref{prop: region_S}, there is some $\lambda\in\mathbb{C}$ with
	\begin{equation*}
	\lambda^2-h\,\lambda+h^{2}\,V^{\prime}(c_{2}(h))[1-e^{\lambda}]=
	\lambda^{2}+\alpha+\beta[1-e^{-\lambda}]=0
	\end{equation*}
	and $\mathrm{Re}(\lambda)>0$. Therefore, the so-called principle of linearized instability, compare, for instance, the first part of Theorem 6.8 in Diekmann et al. \cite[Chapter VII]{Diekmann1995}, shows that the zero solution of Eq. \eqref{eq: DDE_2}, and so the solution $w^{c_2(h), d}$ of Eq. \eqref{eq: DDE-canonical}, is unstable.
\end{proof}

The question about the stability of wavefront solutions with constant-speed of Eq. \eqref{eq: DDE-canonical} which lie on the first branch of Theorem \ref{thm: two_branches} is more sophisticated. Note that for such a solution $w^{c_1(h), d}$, $h^{*}<h<\hat{h}$ and $d\in\R$, necessarily $hV^{\prime}(c_1(h))>1$ holds. Hence, for the corresponding parameter pair $\alpha=-h$ and $\beta=h^2 V^{\prime}(c_1(h))$ we have $\beta+\alpha>0$ such that each of the three cases $(\alpha,\beta)\in S$, $(\alpha,\beta)\in C_1$, and $(\alpha,\beta)\in ((-\infty,0]\times[0,\infty))\setminus \overline{S}$ from Proposition \ref{prop: region_S} may occur. In the last case the linearization has an eigenvalue with positive real part and therefore, similarly to the proof of our last result, the principle of linearized instability shows that $w^{c_1(h), d}$ is unstable. In the other two cases, the linearization does not have any eigenvalues with positive real part but at least one eigenvalue on the imaginary axis. Consequently, in these situations the solution $w^{c_1(h), d}$ may be stable or, more exactly, it has the same local stability properties as the zero solution of the ordinary differential equation obtained from a so-called center manifold reduction. Below we address this issue partially by carrying out a \textit{center manifold reduction} for the case where the eigenvalue $\lambda_0=0$ is simple and the only one in $\sigma_c(G)$. But let us first introduce local center manifolds of Eq. \eqref{eq: DDE_2} at the stationary solution $v(t)=0$, $t\in\R$, in general. 

To begin with, write Eq. \eqref{eq: DDE_2} in the form
\begin{equation}\label{eq: DDE normal}
v^{\prime}(t)=L\,v_{t}+r(v_{t})
\end{equation}
with separated linear part $L=Df(w^{c, d}_t)=D\tilde{f}(0)\in \mathcal{L}(C,\R^{2})$ and the nonlinear part
\begin{equation*}\label{eq: nonlinear}
\begin{aligned}
r(\varphi)&:=\tilde{f}(\varphi)-L\varphi\\
&=\begin{pmatrix}
0\\
h^{2}\,V(c+\varphi_{1}(-1)-\varphi_{1}(0))-h^{2}\,V^{\prime}(c)(\varphi_{1}(-1)-
\varphi_{1}(0))-h\,c
\end{pmatrix}.
\end{aligned}
\end{equation*}
As, regardless of the particular wavefront solution $w^{c, d}$ with constant-speed, we always have $\lambda_0=0\in \sigma_c(G)$, it follows that the center space $C_c\subset C$ is not the zero space but has at least dimension one. Therefore, the center manifold theory for delay differential equations as, for instance, may be found in Diekmann et al. \cite[Chapter IX]{Diekmann1995}, shows the existence of a non-trivial, so-called \textit{local center manifold} $W_c\subset C$ of Eq. \eqref{eq: DDE normal}, or equivalently of Eq. \eqref{eq: DDE_2}, at the stationary solution $v(t)=0$, $t\in\R$. To be more precisely, there exist an open neighborhood $C_{c,0}$ of $0$ in the center space $C_{c}$, an open neighborhood $C_{su,0}$ of $0$ in the stable-unstable space $C_{su}:=C_{s}\oplus C_{u}$, and a continuously differentiable map $w_{c}:C_{c,0}\to  C_{su,0}$ with $w_{c}(0)=0$ and $Dw_{c}(0)=0$ such that
the graph
\begin{equation*}
W_{c}:=\left\lbrace\psi+w_{c}(\psi)\mid\psi\in C_{c,0}\right\rbrace
\end{equation*}
of the so-called \textit{reduction map} $w_c$ has the following properties:
\begin{itemize}
	\item[(i)] $W_{c}$  is a $C^1$- smooth submanifold of $C$ and $\dim W_{c}=\dim C_{c}$;
	\item[(ii)] $W_{c}$ is positively invariant with respect to the semiflow $\tilde{F}$; that is, if $\varphi\in W_{c}$ and $t>0$ such that $\tilde{F}(s,\varphi)\in C_{c,0}\oplus C_{su,0}$ for all $0\leq s\leq t$, then
	\begin{equation*}
	\lbrace \tilde{F}(s,\varphi)\mid 0\leq s\leq t\rbrace\subset W_{c}.
	\end{equation*}
	\item[(iii)] $W_{c}$ contains the segments of all solutions of Eq. \eqref{eq: DDE normal} which are defined on $\R$ and have all their segments in $C_{c,0}\oplus C_{su,0}$.
\end{itemize}

In general, such a local center manifold of a differential equation is not unique. Moreover, in the most cases it is rarely possible to represent the reduction map in a completely explicit way. However, as we show next, in the case of Eq. \eqref{eq: DDE_2} the last point is easily done, provided the linearization of the underlying wavefront solution with constant-speed has only the zero eigenvalue on the imaginary axis and its algebraic multiplicity is one.
\begin{proposition}\label{prop: trivial center-manifold}
	Let $w^{c, d}$, $c>0$ and $d\in\R$, denote a wavefront solution with constant-speed of Eq. \eqref{eq: DDE-canonical} such that $\sigma_c(G)=\lbrace 0\rbrace$ and $h V^{\prime}(c)\not=1$ holds. Then $C_{c}=\R\,\hat{e}_1$ and $w_{c}(\psi)=0$ for all $\psi\in C_{c,0}$.
\end{proposition}
\begin{proof}
	1. Under given assumptions, $\lambda_0=0$ is clearly a simple eigenvalue and the associated one-dimensional eigenspace coincides with the center space $C_c$. Consequently, from Proposition \ref{prop: simple_zero} it trivially follows that $C_{c}=\R\,\hat{e}_1$, and this shows the first part of the assertion.
	
	2. For the proof of the second part of the assertion, let arbitrary $\psi\in C_{c,0}$ be given. Then, by the last part, there is some $k\in\R$ with $\psi=k\,\hat{e}_1$. Now, observe that the function
	\begin{equation*}
	v(t):=\begin{pmatrix}
	k\\0
	\end{pmatrix},\qquad t\in\R,
	\end{equation*}
	is a global solution of Eq. \eqref{eq: DDE normal} and it has the segments $v_{t}=k\,\hat{e}_1=\psi$ as $t\in\R$. In particular, $v_t\in C_{c,0}\oplus C_{su,0}$ for all $t\in\R$. Hence, property (iii) of the associated local center manifold $W_c$ implies that for each $t\in\R$ we have $v_{t}\in W_{c}$, that is, $v_{t}=P_{c} v_{t}+w_{c}(P_{c} v_{t})$ where $P_c$ denotes the continuous projection $P_c$ of $C$ along $C_{su}$ onto the center space $C_c$. All in all, it follows that
	\begin{equation*}
	k\,\hat{e}_1=v_{t}=P_{c}v_{t}+w_{c}(P_{c}v_{t})
	\end{equation*}
	and so
	\begin{equation*}
	C_{c}\ni k\,\hat{e}_1-P_{c}v_{t}=w_{c}(P_{c}v_{t})\in C_{su}.
	\end{equation*}
	Since $C_{c}\cap C_{su}=\lbrace 0\rbrace$ we conclude first that $P_{c}v_{t}=k \hat{e}_1=\psi$ and then $w_c(\psi)=w_c(k\hat{e}_1)=w_{c}(P_{c} v_{t})=0$, which finishes the proof.
\end{proof}
So, under the conditions of the last result, a local center manifold $W_{c}$ of Eq. \eqref{eq: DDE normal} at the stationary solution $v(t)=0$, $t\in\R$, just coincides with the neighborhood $C_{c,0}$ of the origin in the center space $C_{c}$. Furthermore, the dynamics induced by Eq. \eqref{eq: DDE normal} on $W_c$ is the most simplest one:
\begin{proposition}
	Under the assumption of Proposition \ref{prop: trivial center-manifold}, the reduction of Eq. \eqref{eq: DDE normal} to a local center manifold $W_{c}$ is given by the scalar ordinary differential equation
	\begin{equation*}
	p^{\prime}(t)=0.
	\end{equation*}
\end{proposition}
\begin{proof}
	By the center manifold theory, as, for instance, contained in Diekmann et al. \cite{Diekmann1995}, and the last proposition, in the situation considered here the center manifold reduction reads
	\begin{equation*}
	p^{\prime}(t)=Q_{c}\left(p(t)\hat{e}_1+w_{c}(p(t)\hat{e}_1)\right)=
	Q_{c}(p(t)\hat{e}_1+0)=Q_{c}(p(t)\hat{e}_1)
	\end{equation*}
	where $Q_c:C_{c,0}\to\R$ denotes the composition $Q_{c}=\gamma\circ r$ of a linear operator $\gamma:\R^2\to\R$ and the nonlinearity $r$ of Eq. \eqref{eq: DDE normal}. Now, note that for all $p\in\R$ we have $r(p\,\hat{e}_1)=0$. Hence,
	it follows that $Q_{c}(p(t)\hat{e}_1)=0$ and thus $p^{\prime}(t)=0$ as claimed.
\end{proof}

With the statement above we are now in the position to determine the local stability properties of almost all wavefront solutions with constant-speed of Eq. \eqref{eq: DDE-canonical}  along the first branch from Theorem \ref{thm: two_branches}.
\begin{theorem}\label{thm: stability of first branch}
	Given $V$ with  (OVF 1) -- (OVF 4), let $w^{c_1(h), d}$, $h^{*}<h<\hat{h}$ and $d\in\R$, denote a wavefront solution with constant-speed of Eq. \eqref{eq: DDE-canonical} lying on the first branch from Theorem \ref{thm: two_branches}, and let $S\subset (-\infty,0]\times[0,\infty)$ denote the bounded region introduced in Proposition \ref{prop: region_S}. 
	
	Then, if $(-h,h^2 V^\prime(c_1(h)))\in S$ then $w^{c_1(h), d}$ is stable, whereas in the case  $(-h,h^2V^\prime(c_1(h)))\in ((-\infty,0]\times[0,\infty))\setminus\overline{S}$ the solution $w^{c_1(h), d}$ is unstable.
\end{theorem}
\begin{proof}
	Set $\alpha=-h$ and $\beta=h^2 V^\prime(c_1(h))$. Provided $(\alpha,\beta)\in ((-\infty,0]\times[0,\infty))\setminus \overline{S}$, the solution $w^{c_1(h), d}$ is clearly unstable as discussed after Theorem \ref{thm: stability of second branch} and its proof.
	
	Now, assume $(\alpha,\beta)\in S$ and then recall from Proposition \ref{prop: region_S} that in this case $\lambda_0=0$ is a simple eigenvalue of the linearization whereas for all other $\lambda\in\sigma(G)$ we have $\mathrm{Re}(\lambda)<0$. In particular, there is no unstable direction. Therefore, the local center manifold $W_c$ is attractive as, for instance, discussed in Section IX.8 of Diekmann et al. \cite{Diekmann1995}. Consequently, stability assertions for the zero solution of the center manifold reduction carry over to stability assertions for the stationary solution $v(t)=0$, $t\in\R$, of Eq. \eqref{eq: DDE normal}, or equivalently, of Eq. \eqref{eq: DDE_2}. Now,  by the last proposition the center manifold reduction is given by the ordinary differential equation $p^{\prime}(t)=0$, and here the zero solution is clearly stable. This proves the stability of the zero solution of Eq. \eqref{eq: DDE_2}, and thus of solution $w^{c_1(h), d}$ of Eq. \eqref{eq: DDE-canonical}. 
\end{proof}
\begin{remark}\label{rem: bifurcation}
	Observe that the last proposition contains no statement about the stability properties of solution $w^{c_1(h), d}$ with $(-h, h^2 V^\prime(c_1(h)))\in\partial S$. In this case, the associated linearization does not have any eigenvalues with positive real part but, in addition to the simple eigenvalue $\lambda_0=0$, a pair $\pm i\omega$, $\omega>0$, of simple pure imaginary eigenvalues due to Proposition \ref{prop: region_S}. As we will discuss in the next section, it seems that here Eq. \eqref{eq: DDE-canonical} undergoes a degenerate Hopf bifurcation.
\end{remark}

\section{Numerical examples and discussion}
After all the analytical work in the last sections, in the following we consider some numerical examples demonstrating our results. In doing so, we will also briefly address some aspects arising from our simulations. For the numerical calculations we use the solver routine \textit{dde23} of the computing environment MATLAB \cite{MATLAB} with the relative error tolerance of $10^{-9}$ and the absolute error tolerance of $10^{-12}$. The optimal velocity function considered throughout this section is the example $V=V_q$ defined by Eq. \eqref{eq: OVF} with some maximum velocity  $V^{\max}>0$ and safety distance $d_S=0$.

\subsection*{Example 1}\label{subsec: example 1}
In our first example, we consider Eq. \eqref{eq: DDE}, and so Eq. \eqref{eq: DDE-canonical}, for para\-meter $h=h_e:=0.2$ and maximum velocity $V^{\max}=100$. Then, a simple calculation shows that the real
\begin{equation*}
c=c_e:=\frac{h\, V^{\max}}{2}-\sqrt{\frac{(h\, V^{\max})^2}{4}-1}\approx 0.0501
\end{equation*}
satisfies condition \eqref{eq: constant-speed condition}, that is, $c=h\,V(c)$, for the existence of a wavefront solution with constant-speed. Hence, for each $d>0$, the function $w^{c_e, d}:\R\to\R^2$ defined by Eq. \eqref{eq: constant-speed solution} is a solution of Eq. \eqref{eq: DDE-canonical}, and the first component of $w^{c, d}$, that is, $z(t)=-c_e\, t+d$, $t\in\R$, forms a solution of the scalar differential equation \eqref{eq: DDE} whose first derivate is constant and negative.
\begin{figure}[h]
	\includegraphics[width=6cm]{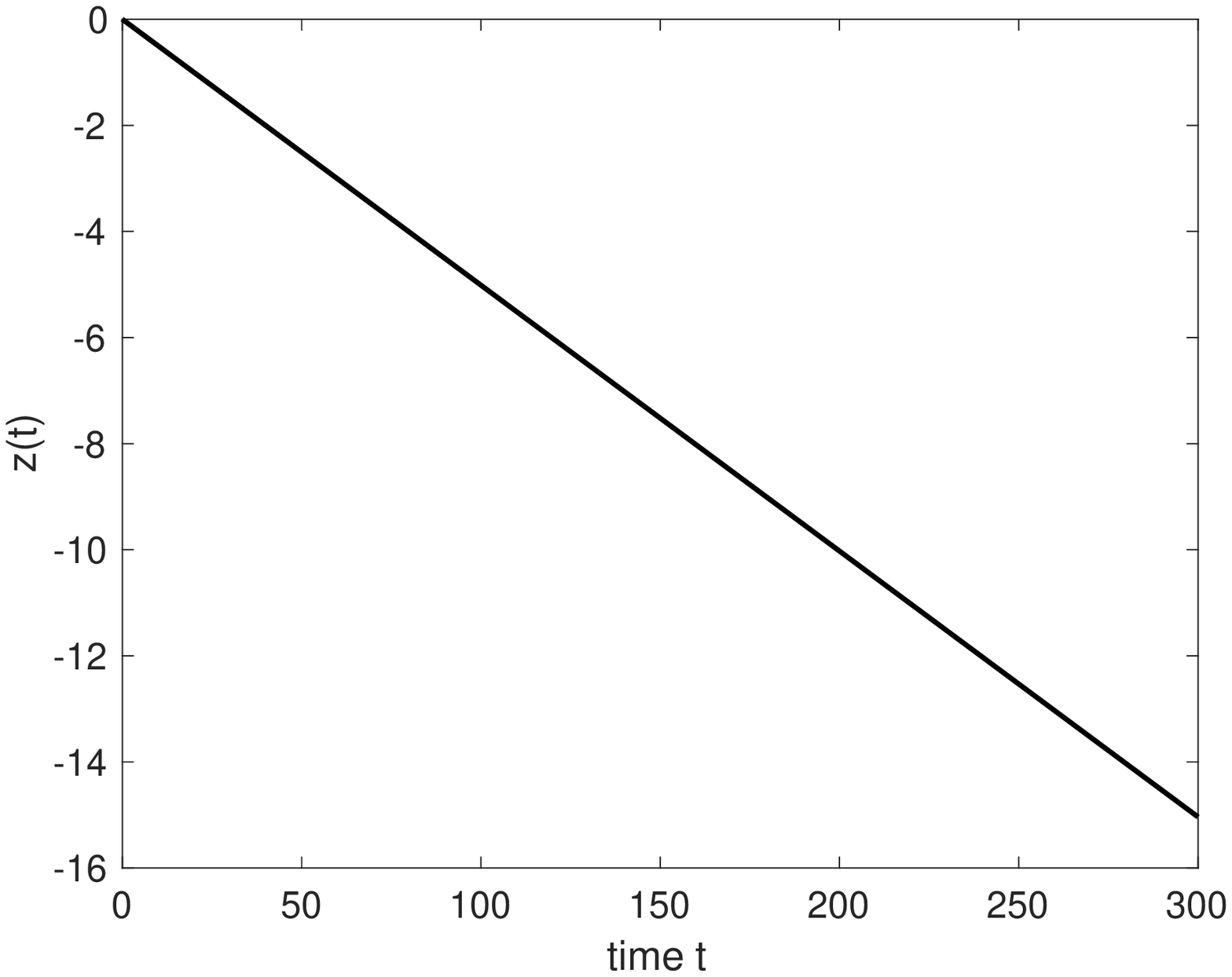}
	\includegraphics[width=6cm]{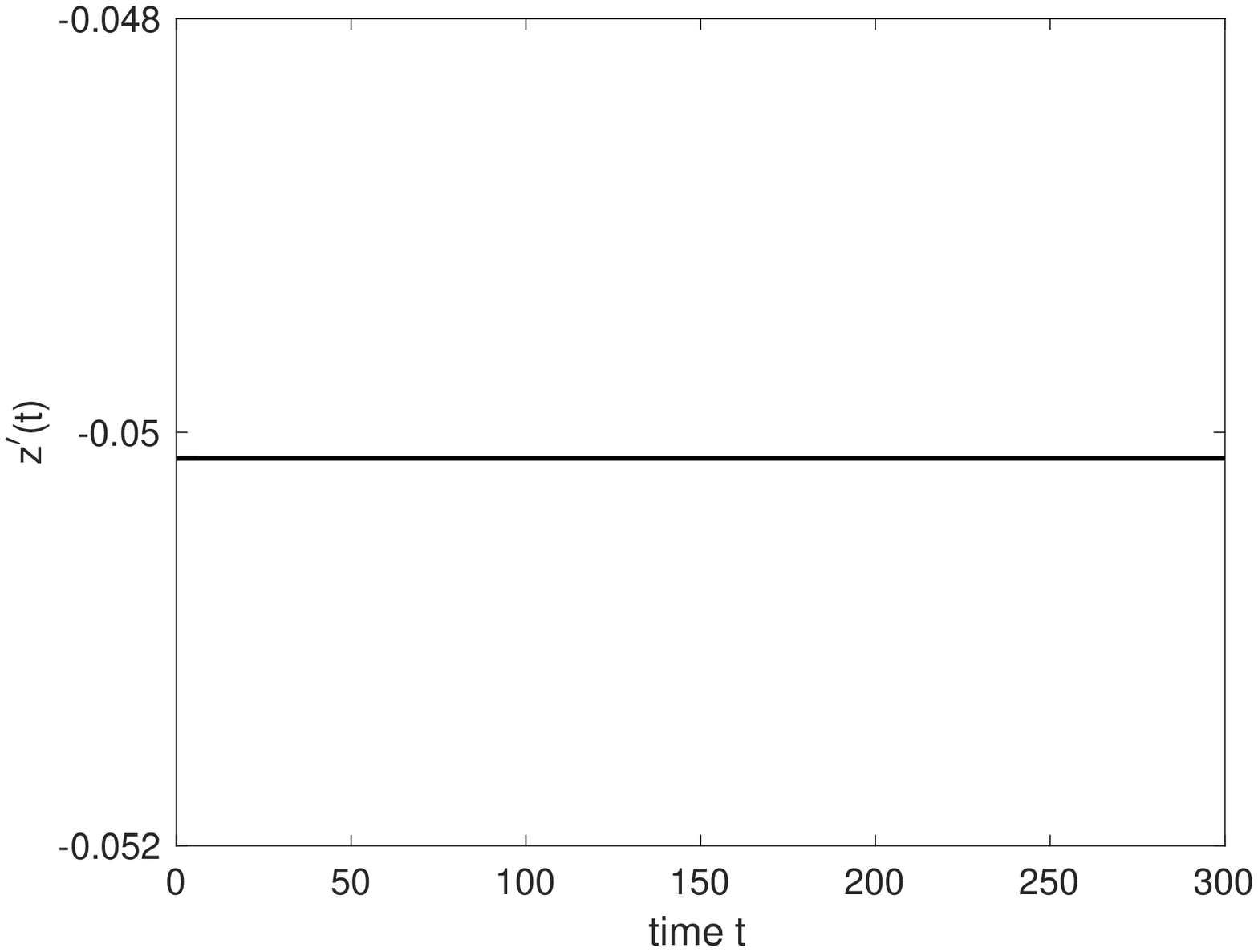}
	\caption{Numerically calculated solution $z$ and its first derivative from Example 1 ($c\approx 0.0501$, $h=0.2$, and $V^{\max}=100$)}\label{fig: example1}
\end{figure}
After fixing $d=0$ and implementing the segment $w^{c_e, 0}_0$ as initial function for the numerical integration, the simulation leads to Figure \ref{fig: example1} which shows the computed solution and its first derivative. 

Numerically, the solution $z$ seems to be stable. For instance, setting $c^*_e=c_e-0.005$ and starting with the initial function $[-1,0]\ni s\mapsto (-c^*_e\,s,-c^*_e)^{T}\in\R^2$ results in the figure below which indicates that the computed solution $z^*$ does not only remain in a small neighborhood of $z$ but actually is attracted by $z$. Indeed, a calculation of the associated stability parameters  from Proposition \ref{prop: region_S} results in $\alpha=\alpha_e:=-0.2$ and $\beta=\beta_e:=h_e^2\,V^\prime(c_e)\approx 0.39899$. Hence, we see at once that $w^{c_e, 0}$, and so solution $z$ of Eq. \eqref{eq: DDE}, is located on the first branch $c_1$ from Theorem \ref{thm: two_branches}, and that, in view of $(\alpha_e, \beta_e)\in S$, it is locally stable due to Theorem \ref{thm: stability of first branch}.
\begin{figure}[h]
	\includegraphics[width=6cm]{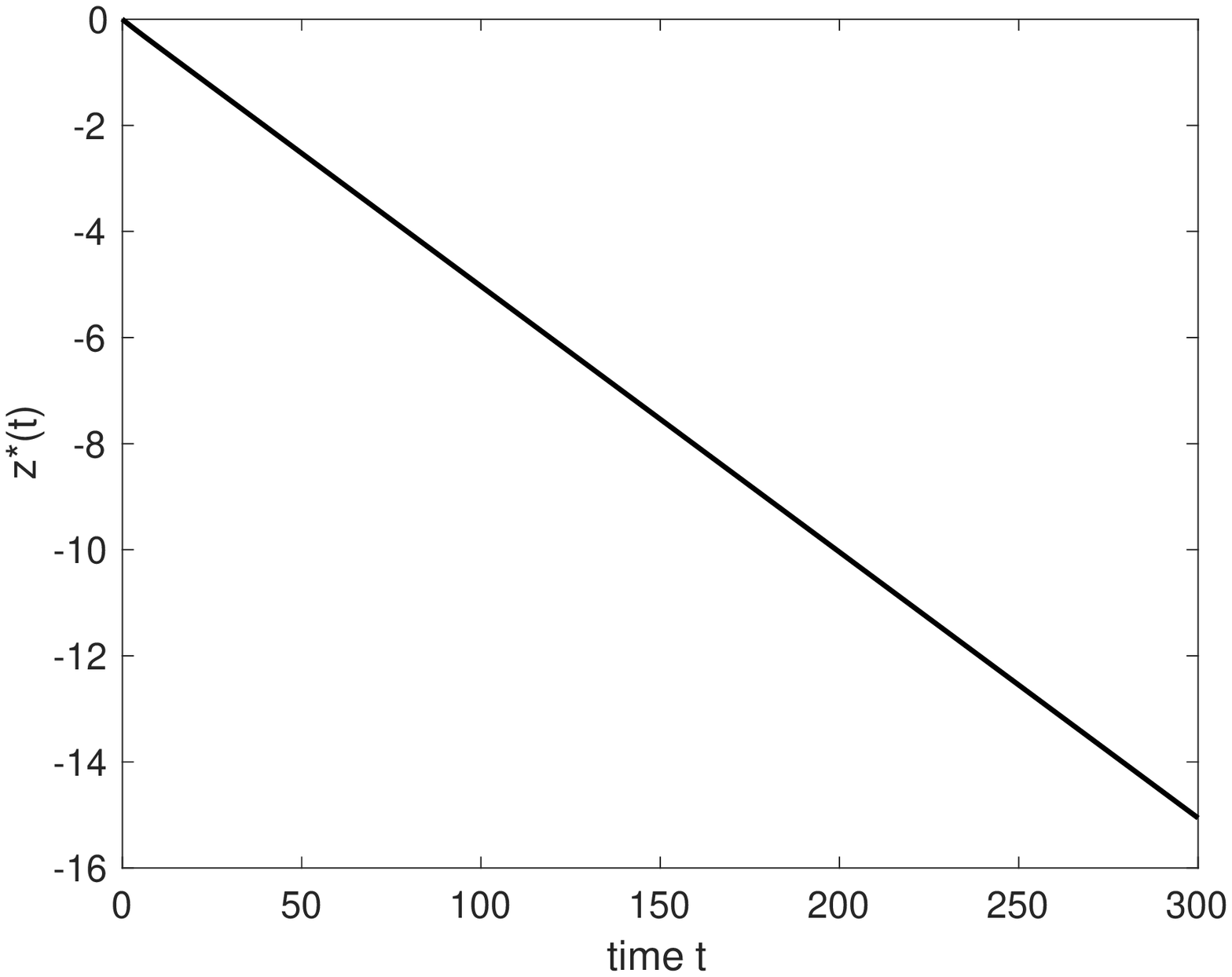}
	\includegraphics[width=6cm]{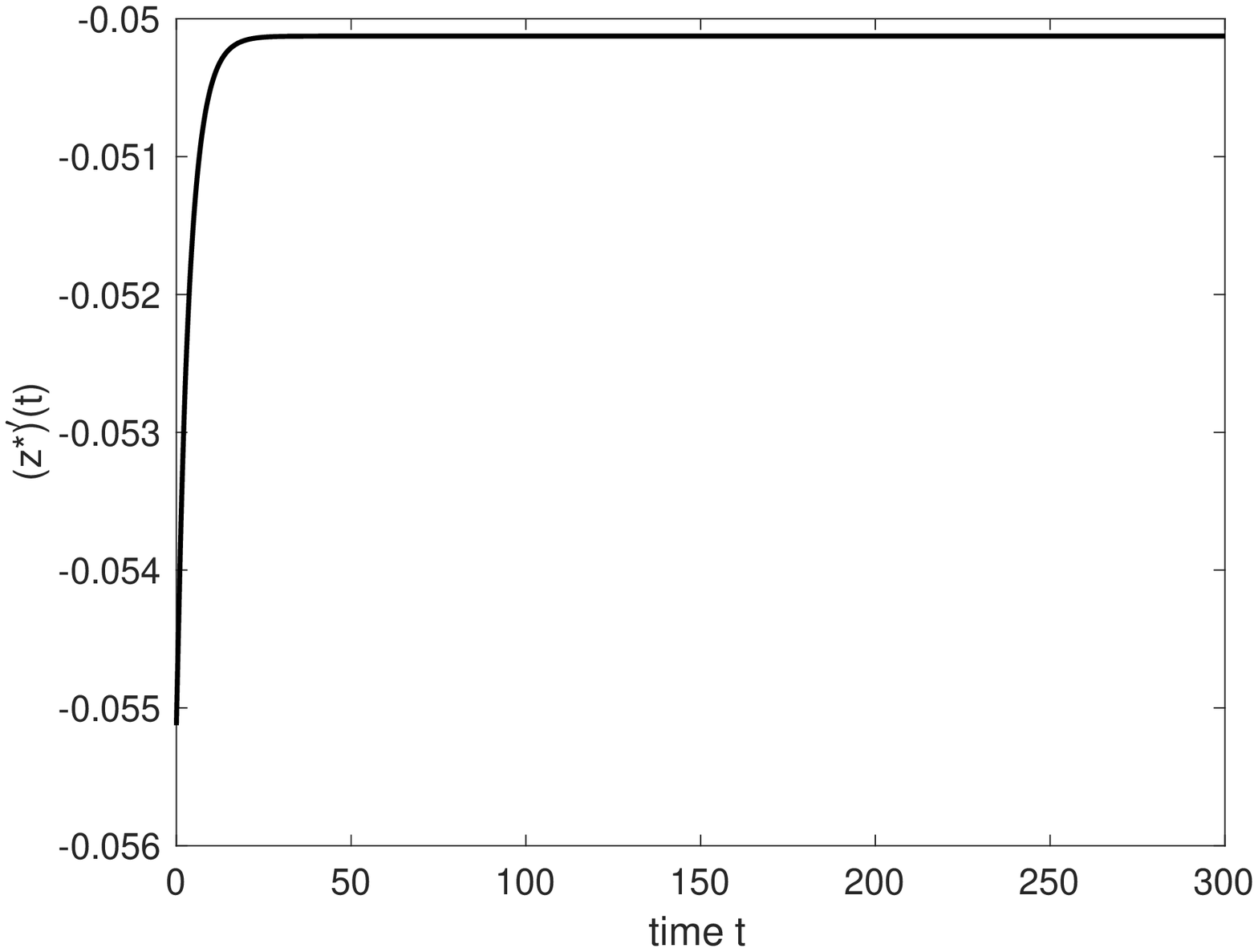}
	\caption{Numerical computation of the disturbed solution $z^*$ and its first derivative from Example 1 ($c^{*}_e\approx 0.0451$)}\label{fig: example1_2}
\end{figure}

\subsection*{Example 2} In this example, we consider Eq. \eqref{eq: DDE-canonical}, and so Eq. \eqref{eq: DDE}, for the same parameter $h=h_e:=0.2$ and the same maximum velocity $V^{\max}=100$ as in the last example. But now we take the real
\begin{equation*}
c=c_e:=\frac{h\, V^{\max}}{2}+\sqrt{\frac{(h\, V^{\max})^2}{4}-1}\approx 19.9499
\end{equation*}
satisfying condition \eqref{eq: constant-speed condition} for the existence of a wavefront solution with constant-speed. Observe that, in consideration of our analysis in Section \ref{sec: existence}, the associated wavefront solution $w^{c_e, 0}$ with constant-speed, and so solution $z=-c_e\,t$, $t\in\R$, of Eq. \eqref{eq: DDE}, necessarily belongs to the second branch $c_2$ from Theorem \ref{thm: two_branches} and thus is unstable due to Theorem \ref{thm: stability of second branch}. Of course, the instability of $z$ is also apparent in numerical simulations. The figure below shows the computed solution for the initial value $w^{c_e, 0}_0$ which theoretically should lead to the solution $z$ for all time under consideration.  
\begin{figure}[h]
	\includegraphics[width=6cm]{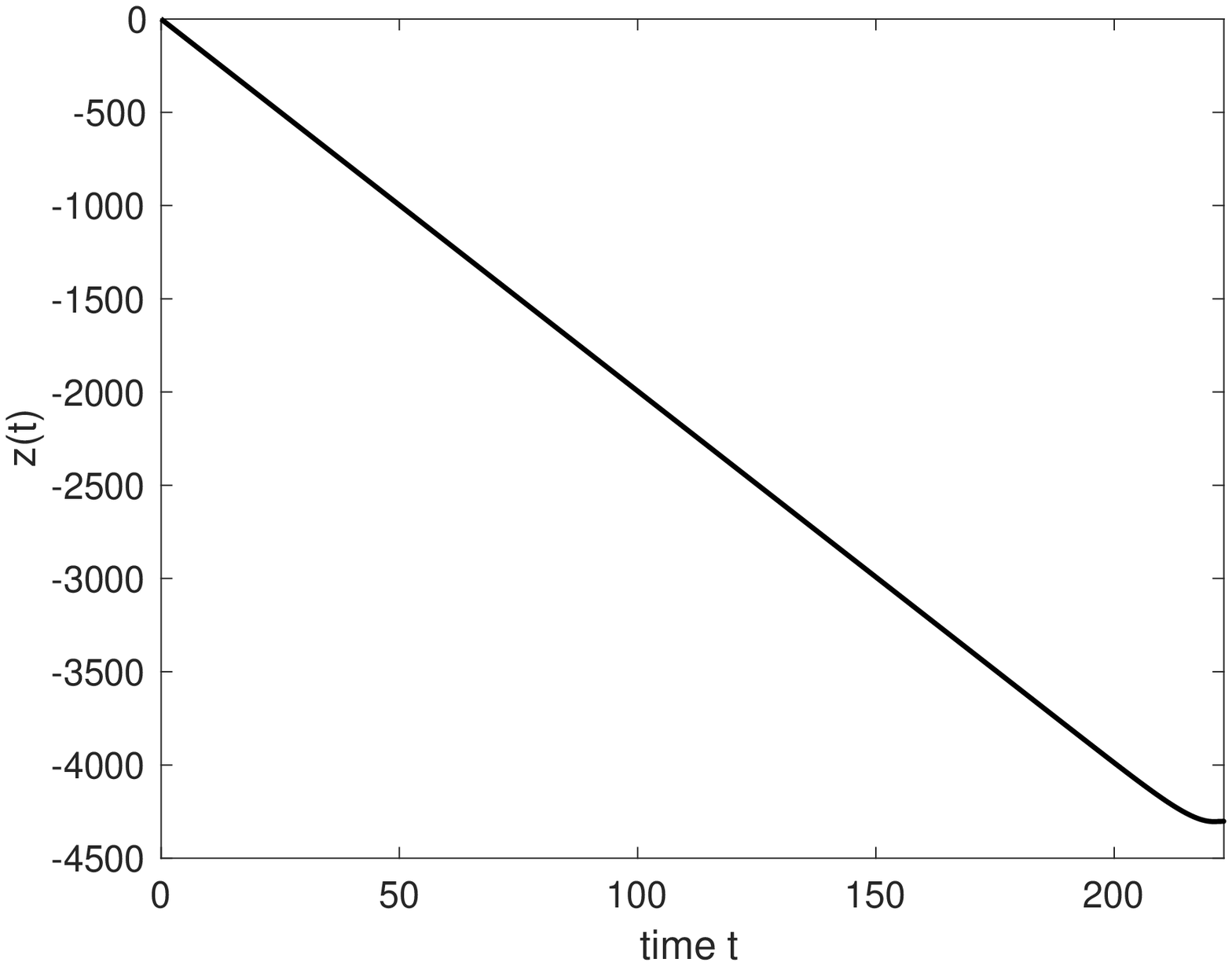}
	\includegraphics[width=6cm]{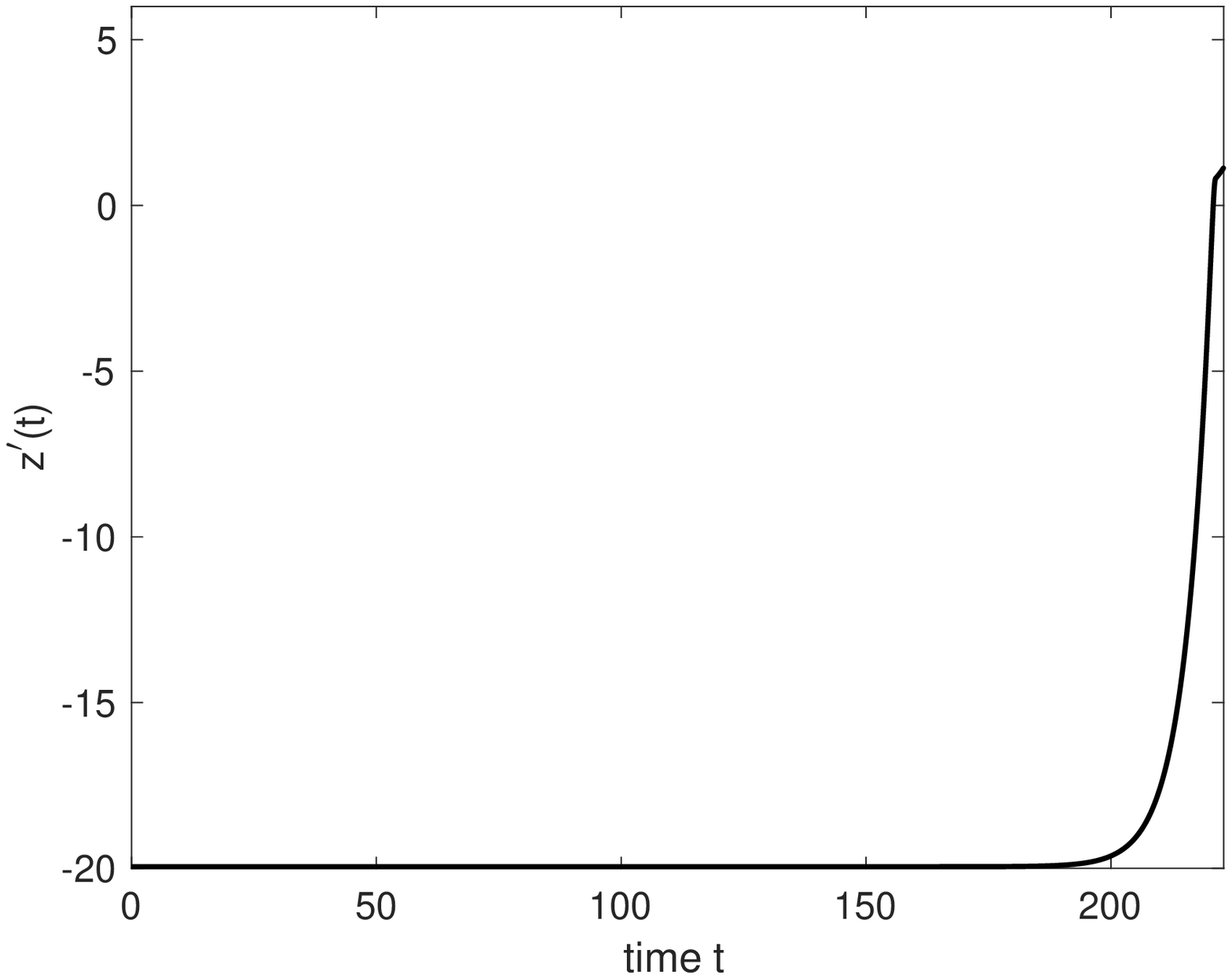}
	\caption{Numerical computation of solution $z$ and its first derivative from Example 2 ($c\approx 19.9499$, $h=0.2$, and $V^{\max}=100$)}\label{fig: example2}
\end{figure}
It seems, also supported by the scale of the axes, that at first the numerically computed solution coincides with $z$ for a (long) while as expected.  However, finally we end up with something else. And the reason here is the interplay between the instability of the solution $z$ and the rounding in the floating point arithmetic. To be more precisely, at some time, the rounding in the floating point arithmetic first leads to the fact that the computed solution leaves the quasi-stationary case and  ``jumps'' to some other orbit of Eq. \eqref{eq: DDE} in the immediate vicinity of the orbit of $z$. Then, the instability of the quasi-stationary solution $z$ ``forces'' the numerical computed solution to leave all sufficiently small neighborhoods of $z$. At the end, the simulation shown in Figure \ref{fig: example2} does not meet the solution $z$ of Eq. \eqref{eq: DDE} subject to the initial function $w_0^{c_e, 0}$. Moreover, it is irrelevant with respect to the car-following model given by Eq. \eqref{eq: traffic_model} as the computed solution is clearly not strictly decreasing at all (but strictly increasing on some interval with length greater than $2h$).

\subsection*{Example 3}
In our final example, we consider again the situation of an unstable wavefront solution with constant-speed but this time for parameter $h=h_e:=1.5$ and maximum velocity $V^{\max}=2.841$. Then,
\begin{equation*}
c=c_e:=\frac{h\, V^{\max}}{2}-\sqrt{\frac{(h\, V^{\max})^2}{4}-1}\approx 0.2492
\end{equation*}
fulfills $h_e\,V(c_e)=c_e$ such that $w^{c_e, 0}$ forms a wavefront solution with constant-speed of Eq. \eqref{eq: DDE-canonical}, and, accordingly, $z(t)=-c_e\,t$, $t\in\R$, a quasi-stationary solution of Eq. \eqref{eq: DDE}. 
\begin{figure}[h]
	\includegraphics[width=6cm]{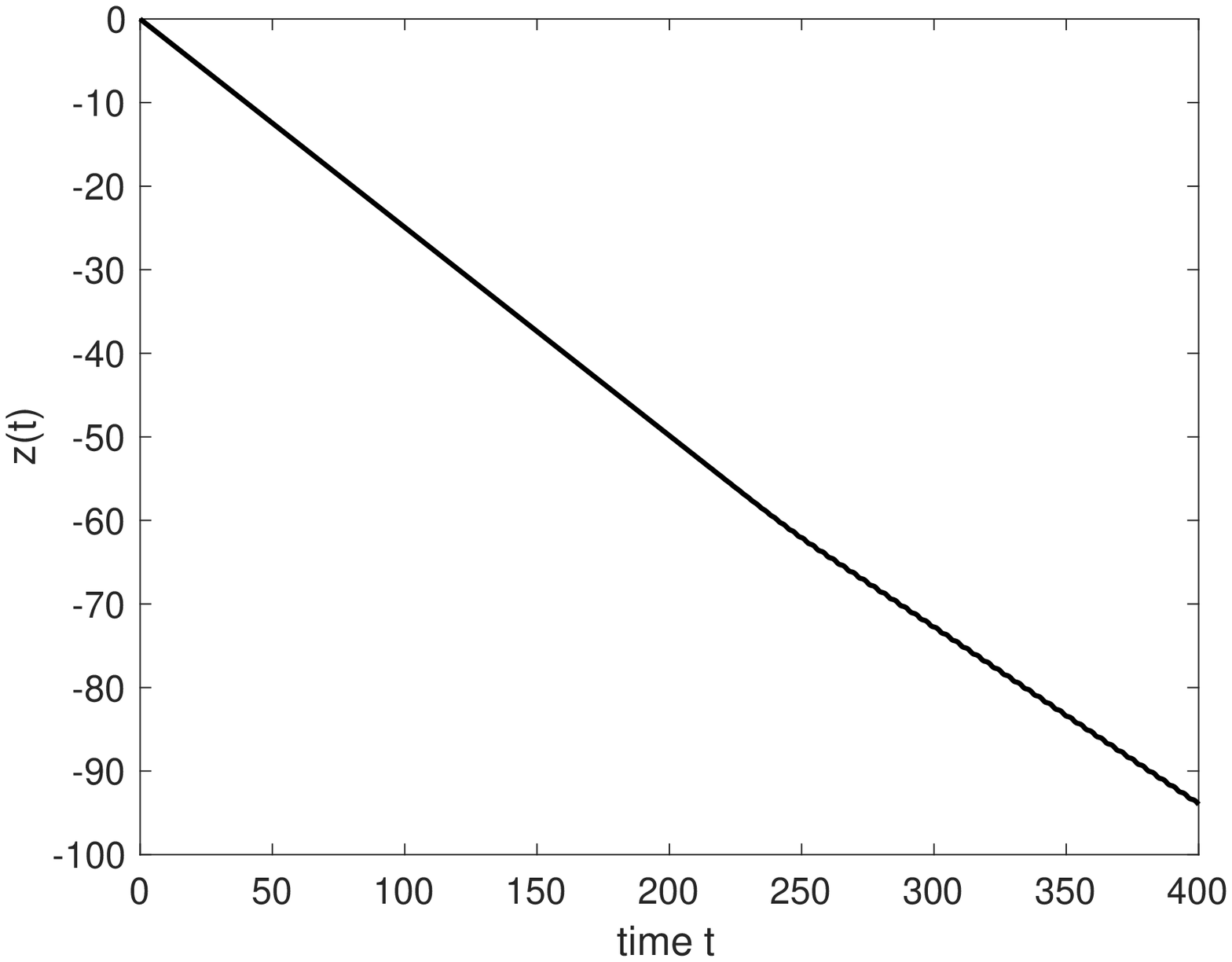}
	\includegraphics[width=6cm]{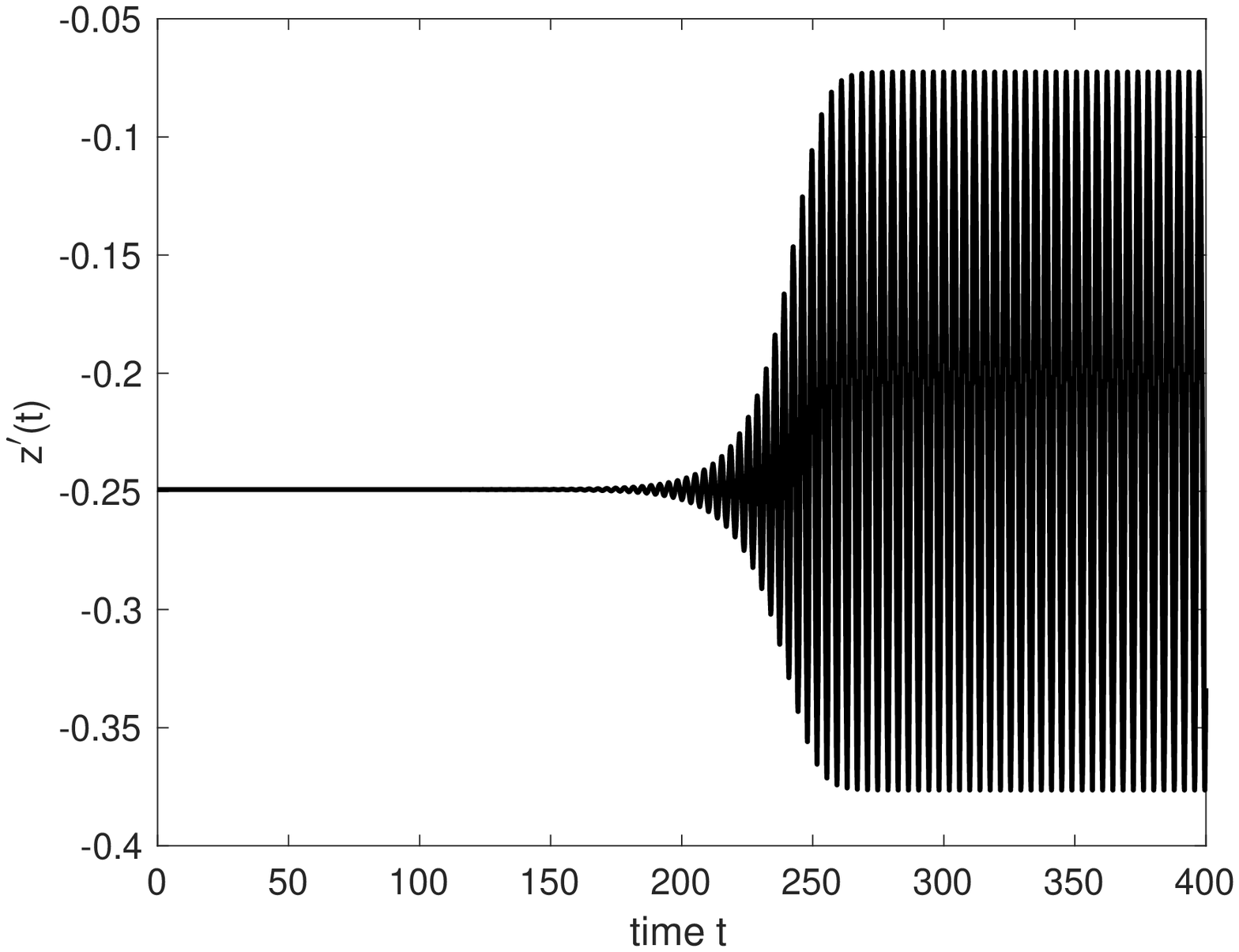}
	\caption{Numerical computation of solution $z$ and its first derivative from Example 3 ($c\approx 0.2492$, $h=1.5$, and $V^{\max}=2.841$)}\label{fig: example 3}
\end{figure}
The corresponding stability parameters are $\alpha=\alpha_e:=-1.5$ and $\beta=\beta_e:=h_e^2\, V(c_e)\approx 2.8245$, and it is easily seen that the solution under consideration belongs to the first branch $c_1$ from Theorem \ref{thm: two_branches}, and that, in view $(\alpha_e,\beta_e)\not\in \overline{S}$, it is unstable due to Theorem \ref{thm: stability of first branch}. We choose $w^{c_e, 0}_0$ as initial function and compute the solution $z$ numerically. The result of this computation is shown in Figure \ref{fig: example 3}, and, apparently, it is completely different in nature as in the example before. At the initial stage of the simulation the computed solution seems, similarly to the last example, to coincides with $z$. But then, caused by the rounding in the floating point arithmetic and the instability of $z$, the computed solution leaves the quasi-stationary state, and its first derivative begins to oscillate about the value $-c_e$ with decreasing minimal and increasing maximal value. After reaching some thresholds for the minimal and maximal value, the oscillation becomes completely regular such that, in the final stage of the simulation, the computed solution is uniform, and its first derivative not only uniform but periodic. Compare here also Figure \ref{fig: example 3_2} showing the final stage of the numerical computation.
\begin{figure}[h]
	\includegraphics[width=6cm]{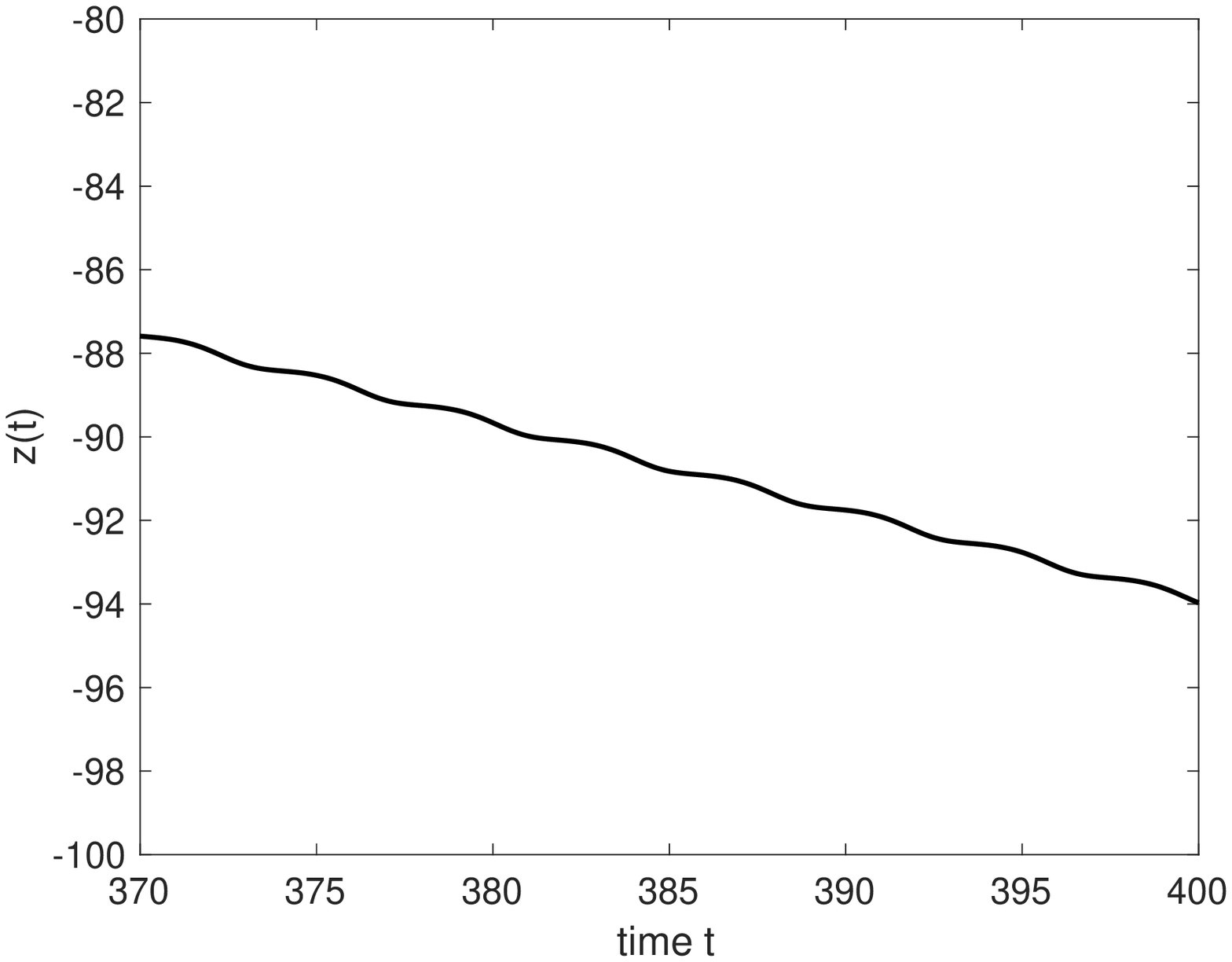}
	\includegraphics[width=6cm]{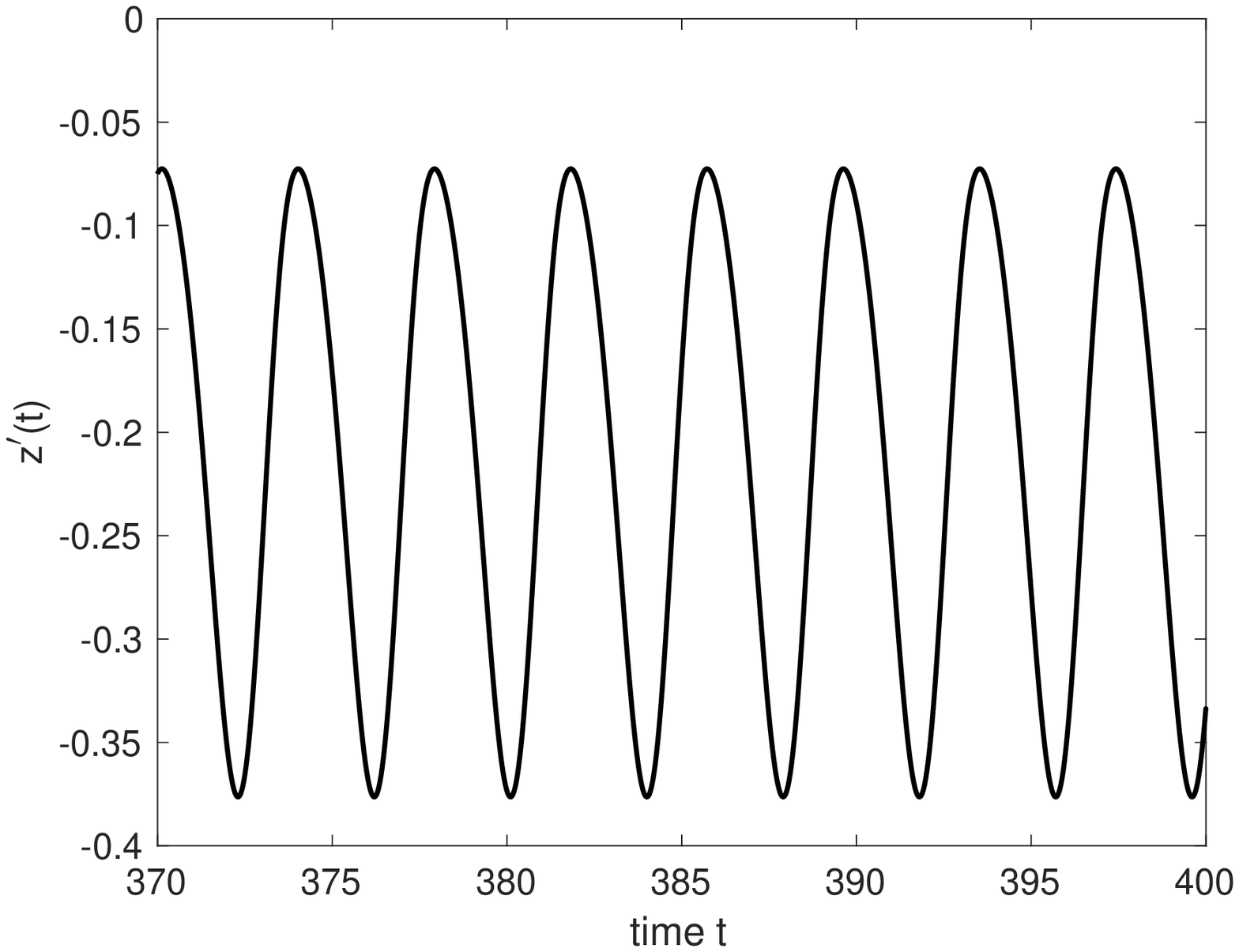}
	\caption{The final stage of the numerical computation of solution $z$ and its first derivative from Example 3 ($c\approx 0.2492$, $h=1.5$, and $V^{\max}=2.841$)}\label{fig: example 3_2}
\end{figure}

Most likely, the above figure shows a real solution of Eq. \eqref{eq: DDE}, but, of course, not subject to the initial value $w^{c_e, 0}_0$. A rough explanation for what seemingly happens here was indicated in Remark \ref{rem: bifurcation}. But let us specify it more precisely by the following conjecture which shall be addressed analytically in \cite{Stumpf2016}.
\begin{conjecture}\label{con}
	Given $C^2$-smooth $V$ satisfying  (OVF 1) -- (OVF 4), let $w^{c_1(h_H), d}$, $h^{*}<h_H<\hat{h}$ and $d\in\R$, denote a wavefront solution with constant-speed of Eq. \eqref{eq: DDE-canonical} lying on the first branch from Theorem \ref{thm: two_branches}, and let $S\subset (-\infty,0]\times[0,\infty)$ denote the bounded region introduced in Proposition \ref{prop: region_S}. Further, suppose that, for the associated stability parameters of solution $w^{c_1(h_H), d}$, it holds that $(-h_H,h_H^2\,V^{\prime}(c_1(h_H)))\in\partial S$.
	
	Then, at parameter value $h=h_H$ and solution $w=w^{c_1(h_H, d)}$, Eq. \eqref{eq: DDE-canonical}, and so Eq. \eqref{eq: DDE} at parameter $h=h_H$ and solution $z=-c_1(h_H)\, t+d$, undergoes a degenerate Hopf-bifurcation which is supercritical. The bifurcating solutions are not periodic but their first derivatives.
\end{conjecture} 

Returning to our example under consideration, we note that the branch $c_1=c_1(h)$ of wavefront solutions $w^{c_1(h), 0}$ with constant-speed is at least defined for all $h^{\star}<h\leq h_e$ with $h^{\star}=2/V^{\max}\approx 0.8127$. Next, after fixing some $h_f>h^{\star}$ with $h_f-h^{\star}>0$ sufficiently small, a simple argument shows that the associated stability parameters 
\begin{equation*}
\alpha=\alpha:=-h_f\qquad \text{and}\qquad \beta=\beta_f:=h_f^2\,V^\prime(c_1(h_f))
\end{equation*} of solution $w^{c_1(h_f), 0}$ form a point inside the region $S$ from Proposition \ref{prop: region_S}. Thus, $w^{c_1(h_f), 0}$ is stable due to Theorem \ref{thm: stability of first branch}. Now, let us increase the parameter value $h$ continuously from $h_f$ to $h_e$. At first, all the solutions $w^{c_1(h), 0}$ remain stable as the associated parameters $(\alpha(h),\beta(h)):=(-h,h^2\,V^{\prime}(c_1(h))))$ are contained inside $S$. On the other hand, the curve $\mathcal{C}: [h_f,h_e]\ni h\mapsto (\alpha(h),\beta(h))$ has to leave and stay outside of $S$ for all sufficiently large $h\leq h_e$, since we have $(\alpha(h_e),\beta(h_e))=(\alpha_e,\beta_e)$ and do already know that $(\alpha_e,\beta_e)\in ((-\infty,0]\times[0,\infty))\backslash \overline{S}$. Therefore, the curve $\mathcal{C}$ has to cross the boundary $\partial S$ of $S$ at some point $h_f<h_H<h_e$. Moreover, it is easily seen that the curve $\mathcal{C}$ has to do so by crossing the curve $C_1$ from Proposition \ref{prop: region_S} which particularly shows that by increasing $h$ about the value $h_H$ a pair of simple complex conjugate eigenvalues of the linearization moves from the left to the right half-plane of $\mathbb{C}$. For that reason, by increasing $h$ from $h_f$ to $h_e$ we loose the stability of the associated solution $w^{c_1(h), 0}$ at the value $h=h_H$. But, as conjectured, that is done by undergoing a supercritical Hopf bifurcation such that, for each parameter $h>h_H$ with $h-h_H>0$ sufficiently small, we find a locally stable solution $w^{h}_H$ of Eq. \eqref{eq: DDE-canonical} which is not periodic but its first derivative.

With the above in mind, let us briefly revisit the numerical simulation in this example. As already said, the solution $z$ is unstable. On the other hand, it seems that the bifurcating branch of solutions $w_H^h$ is even defined for the parameter value $h=h_e$. In fact, most likely, Figure \ref{fig: example 3_2} namely shows, the solution $w^{h_e}_H$ which is locally stable. So, after having left the quasi-stationary state of $z$ due to the rounding in the floating point arithmetic and the instability, the computed solution seems first to be attracted by $w^{h_e}_H$, and then, after sufficiently long time, to coincide, more or less, with $w^{h_e}_H$ as indicated in Figure \ref{fig: example 3}.

Finally, observe that the example discussed here is also significant for the traffic model described by Eq. \eqref{eq: traffic_model} as it suggests the existence of wavefront solutions with stop-and-go behavior. Indeed, a bifurcating solution of Eq. \eqref{eq: DDE-canonical} from Conjecture \ref{con} leads to a solution of the traffic model where each driver accelerates and brakes alternately.


\end{document}